\documentclass[journal, letterpaper]{IEEEtran}
\usepackage{amsthm, amssymb, amsmath}
\usepackage{verbatim, float}
\usepackage{mathrsfs}
\usepackage{graphics}
\usepackage[usenames,dvipsnames]{pstricks}
\usepackage{epsfig}
\usepackage{pst-grad} 
\usepackage{pst-plot} 
\usepackage{cases}
\usepackage{algorithmic}

\usepackage[top=1.4cm, bottom=1.4cm, left=1.33cm, right=1.33cm]{geometry}
\usepackage{url}

\theoremstyle{plain}
\newtheorem{theorem}{Theorem}
\newtheorem{lemma}{Lemma}

\newtheorem{corollary}{Corollary}
\newtheorem{proposition}{Proposition}

\theoremstyle{definition}
\newtheorem{definition}{Definition}
\newtheorem{example}{Example}
\newtheorem{remark}{Remark}

\newcommand{\A}{{\mathcal A}}
\newcommand{\B}{{\mathcal B}}
\newcommand{\C}{{\mathcal C}}

\newcommand{\E}{{\mathcal E}}

\newcommand{\G}{{\mathcal G}}
\renewcommand{\H}{{\mathcal H}}

\newcommand{\K}{{\mathcal K}}
\renewcommand{\L}{{\mathcal L}}

\newcommand{\X}{{\mathcal X}}

\newcommand{\V}{{\mathcal V}}

\renewcommand{\S}{{\mathcal S}}
\renewcommand{\P}{{\mathcal P}}
\renewcommand{\O}{{\mathcal O}}

\newcommand{\bp}{{\boldsymbol{p}}}

\newcommand{\al}{\alpha}
\newcommand{\be}{\beta}

\newcommand{\lam}{\lambda}
\newcommand{\si}{\sigma}

\newcommand{\cC}{{\mathscr{C}}}

\newcommand{\cR}{{\mathscr{R}}}

\newcommand{\define}{\stackrel{\mbox{\tiny $\triangle$}}{=}}

\newcommand{\et}{{\emph{et al.}}}

\newcommand{\toG}{{\theta_1}(\G)}
\newcommand{\ttG}{{\theta_2}(\G)}
\newcommand{\tcG}{{\theta^{\text{c}}(\G)}}
\newcommand{\tto}{{\theta_1}}
\newcommand{\ttt}{{\theta_2}}

\newcommand{\ttc}{{\theta^{\text{c}}}}
\newcommand{\knn}{{\mathcal{K}_{n,n}}}
\newcommand{\knr}{{\mathcal{K}_{n^{[r]}}}}

\newcommand{\AB}{{|\A||\B|}}

\newcommand{\pr}{{\text{Prob}}}

\begin{document}

\title{Latent Network Features and Overlapping Community Discovery via Boolean Intersection Representations
\thanks{%
S. H. Dau and O. Milenkovic are with the Coordinated Science Laboratory, University of Illinois at Urbana-Champaign, 1308 W. Main Street, Urbana, IL 61801, USA. Emails: \{hoangdau, milenkov\}@illinois.edu.}
} 
\author{Son Hoang Dau, \emph{Member}, and Olgica Milenkovic, \emph{Senior Member}, \emph{IEEE}}

\date{}
\maketitle
\pagestyle{empty}
\vspace{-0.1in}
\begin{abstract}
We propose a new latent Boolean feature model for complex networks that captures different types of node interactions and network communities. 
The model is based on a new concept in graph theory, termed the Boolean intersection representation of a graph, which generalizes
the notion of an intersection representation. We mostly focus on one form of Boolean intersection, termed \emph{cointersection}, and 
describe how to use this representation to deduce node feature sets and their communities. We derive several general bounds on the minimum number of features used in cointersection representations and discuss graph families for which exact cointersection characterizations are possible. Our results also include algorithms for finding optimal and approximate cointersection representations of a graph. 
\end{abstract}
\vspace{-0.05in}
\section{Introduction}
\label{sec:intro}
An important task in network analysis is to understand the mechanism behind the formation of a given complex network. 
\emph{Latent feature models} for networks 
seek to explain the observed pairwise connections among the nodes in a network by associating to each node a set of features
and by setting rules based on which pairs of nodes are connected according to their features. 
Inference of latent network features not only allows for the discovery of community 
structures in networks via association with features but also aids in predicting unobserved connections. As such, feature inference is invaluable in the study of social networks, protein complexes and gene regulatory modules.

Probabilistic latent feature models for networks are usually studied via machine learning techniques; known problems and analytic approaches include the Binary Matrix Factorization model~\cite{MeedsGharamaniNealRowisNIPS2006}, the Mixed-Membership Stochastic Block model~\cite{AiroldiBleiFienbergXing2008}, the Infinite Latent Feature/Attribute model~\cite{MillerJordanGriffithsNIPS2009, PallaKnowlesGhahramaniICML2012}, the Multiplicative
Attribute Graph model~\cite{KimLeskovec2011}, the Attribute Graph Affiliation model~\cite{YangLeskovecICDM2012}, and the Cluster Affiliation model (or BIGCLAM)~\cite{YangLeskovecWSDM2013}. In contrast, almost nothing is known about deterministic, combinatorial latent feature models.  

In the recent work of Tsourakakis~\cite{TsourakakisWWW2015},
a probabilistic latent feature model for networks was proposed
that implicitly uses the notion of intersection representations of
graphs~\cite{ErdosGoodmanPosa1966,Hefner1991, ChungWest1994} and builds upon the overlapping community detection approach of
Bonchi {\et}~\cite{BonchiGionisUkkonenICDM2011}. 
More specifically, in this model one fixes the total number of features and
tries to assign to each vertex a subset of features in a way that maximizes a certain \emph{score}.
Here, the score of a specific feature assignment is the count of unordered pairs of vertices $(u,v)$
that satisfies the so-called \emph{Intersection Condition}, which states that $u$ and $v$ are adjacent
if and only if they share at least one common feature. 
In particular, if one insists on a \emph{perfect} score, i.e., a score equal to $\binom{n}{2}$, 
then the minimum number of features required reduces to the \emph{intersection number} of the graph~\cite{ErdosGoodmanPosa1966}. 
An assignment of sets of features to vertices that achieves the perfect score is known as an \emph{intersection representation} of a graph (see Fig.~\ref{fig:intersection})\footnote{The intersection representation of graph arises in numerous problems such as the
keyword conflict problem, the traffic phasing problem, and the competition graphs from food
webs, to name a few, and has been extensively studied in the literature (see, for instance~\cite{Pullman1983, Roberts1985}).}. 
If in the Intersection Condition one insisted on $u$ and $v$ sharing at least $p \geq 1$ common features, achieving a perfect score would require a minimum number of features equal to the $p$-intersection number of the graph~\cite{Hefner1991, ChungWest1994}. 
\vspace{-10pt}
\begin{figure}[H]
\centering
\includegraphics[scale=0.85]{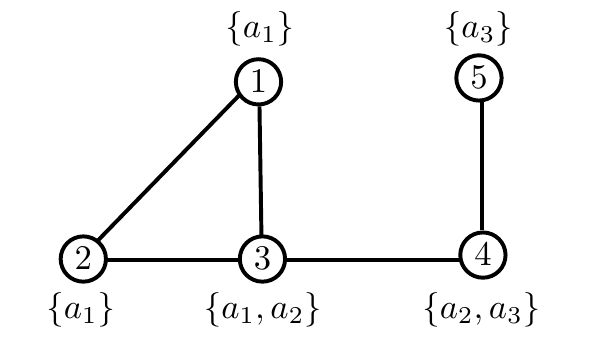}
\caption{Illustration of an intersection representation of a graph from~\cite{TsourakakisWWW2015}. Vertices are assigned
subsets from the feature set $\A = \{a_1,a_2,a_3\}$ so that two vertices are adjacent if and only if 
they share at least one common feature. In this case, the intersection number is three.}
\label{fig:intersection}
\vspace{-5pt}
\end{figure}
Intersection representations elucidate \emph{overlapping community structures} via a simple generative principle: one feature - one community. As an illustrative example, each feature in Fig.~\ref{fig:intersection} may describe one community; the triangle forms one community defined by feature $a_1$, and the remaining two edges are defined by features $a_2$ and $a_3$, respectively. Note that all communities are cliques, and that they may overlap (intersect).

We propose to extend the combinatorial variant of the model studied by Bonchi {\et}~\cite{BonchiGionisUkkonenICDM2011} and by 
Tsourakakis~\cite{TsourakakisWWW2015} to a much more general setting by using Boolean functions of features that can express more complicated interactions among nodes (vertices). For instance, suppose that there are three different types of features, namely `Family member', `City', and `Hobby'. The Boolean function $f(x_1, x_2, x_3) =x_1 \vee (x_2 \wedge x_3)$ can be used to express the connection rule that two people are Facebook friends if and only if either they are family members or they have lived in at least one common city and shared at least one common hobby. As such, it asserts that the `Family' feature is more relevant than either of the `City' or `Hobby' features. More generally, we can use any Boolean function $f=f(x_1,\ldots,x_r)$ together with a vector $\bp = (p_1,\ldots,p_r)$, $p_i \geq 1$,
to describe a connectivity rule based on $r$ different types of features in which the requirement `sharing at least one common feature 
of type $\A_i$' is replaced by the requirement `sharing at least $p_i$ common features of type $\A_i$'.

In the scope of this paper, we mostly focus on a basic building block of Boolean functions, namely the AND function of two variables $f(x_1,x_2) = x_1 \wedge x_2$. It is straightforward to see that the Boolean OR function leads to results identical to those obtained for the simple intersection problem, and results obtained for AND functions allow one to easily extend all the proposed approaches to the case of Boolean functions that include both AND and OR operations.
For simplicity, we also consider $(p_1,p_2) = (1,1)$. 
To illustrate the latent feature model arising in this setup, we consider the example in Fig~\ref{fig:illustration}. 
The network has five nodes, which represent five different people. 
Each person is assigned \emph{two} distinct sets of features, one representing the hobbies
that the person has and the other representing the cities that the person has lived in.
For instance, let $\A = \{a_1, a_2\}$ be such that $a_1$ stands for \emph{fishing} and $a_2$ stands for \emph{playing soccer}, 
and let $\B = \{b_1, b_2\}$ be such that $b_1$ stands for \emph{Hanoi} and $b_2$ stands for \emph{Champaign}. 
Then Person $4$ is assigned two sets of features, namely $\{a_2\}$ and $\{b_1,b_2\}$, which states that this person
has soccer as a hobby and has lived in both Hanoi and Champaign (to avoid notational clutter, we use $\{{a_2 \, | \, b_1,b_2\}}$ to denote pairs of sets).
\vspace{-8pt}
\begin{figure}[htb]
\centering 
\includegraphics[scale=0.85]{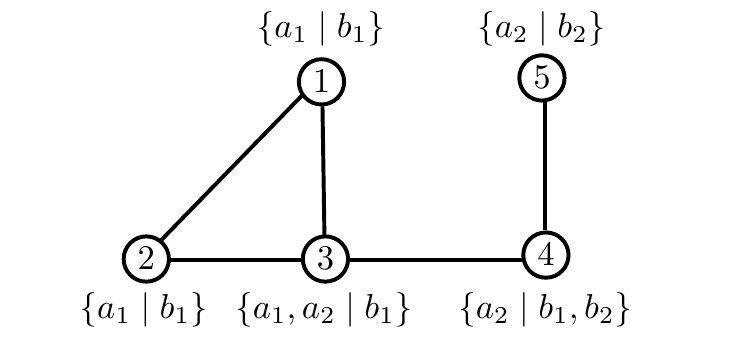}
\caption{Each node is assigned a set of features from $\A=\{a_1,a_2\}$ and a set of
features from $\B = \{b_1,b_2\}$. Two nodes are connected by an edge if and only if
they share at most one feature from $\A$ and one feature from $\B$.} 
\label{fig:illustration}
\vspace{-5pt}
\end{figure}
Suppose that two people are connected if and only if they share at least one common hobby AND they have lived in at least one common city.  
For instance, Person $3$ and Person $4$ are connected because they have soccer as a common hobby and they both have lived in Hanoi. However, Person $3$ and Person $5$ are not connected, even though they both like playing soccer, because they have not lived in the same city. 

Given the nodes' corresponding sets of features and the rules as of how to connect two nodes,
it is clear how the graph emerges. The problem of interest is the opposite: 
\emph{under the assumption that the graph is given and that each node is assigned two subsets of features from $\A$ and $\B$, 
where $\A$ and $\B$ are two disjoint sets of features, and that two nodes are connected
if and only if they share at least one feature from $\A$ and at least one feature from $\B$, how can we infer the latent features assigned to the nodes?} Usually, the latent
features are abstracted as elements from a discrete set, and the mapping between the elements and the real features is determined based on available data. 

Our first aim is to determine the smallest
possible number of features $\min (|\A|+|\B|)$ needed to explain a given graph. 
We refer to this quantity as the \emph{cointersection number} of a graph. 
Note that the notions of cointersection number and cointersection representation of graphs have not been studied before in the literature.
We then proceed to establish general lower and upper bounds on the cointersection number of a graph via its intersection number. In addition, we derive several explicit bounds for some particular families of graphs,
including bipartite graphs, multipartite graphs, and graphs with bounded degrees
(Section~\ref{sec:upper_bounds}). 
We also examine the tightness of these bounds (Section~\ref{sec:tightness}). In particular, we describe an interesting connection between the cointersection representations of certain complete multipartite graphs and affine planes. We provide an exact algorithm to find an 
optimal cointersection representation of a graph by using SAT solvers (Section~\ref{subsec:SAT}).
We also develop a randomized algorithm to find an approximate cointersection representation of a graph
in Section~\ref{subsec:MCMC}. 
Finally, we extend the bounds on
the cointersection number for the case when a general Boolean function is used instead of
the AND function (Section~\ref{sec:Boolean}). 

As a parting remark, we point out that there exist many other applications of latent feature modeling which pertain to communication networks, spectrum allocation being one particular example of interest. We defer the discussion of these topics to a companion paper.

\section{Preliminaries}
\label{sec:pre}
We start by formally introducing our new latent feature model and describing its relevant properties.

\subsection{The cointersection Model}

\begin{definition}
\label{def:cir}
Let $\A$ and $\B$ be two disjoint nonempty subsets of features of cardinalities $\al$ and $\be$, 
respectively. 
An $(\al \mid \be)$\emph{-cointersection representation} (CIR) for a graph $\G = (\V,\E)$ is a family
$\cR = \{(A_v \mid B_v): v \in \V\}$, where $A_v \subseteq \A$, $B_v \subseteq \B$, that satisfies 
the so-called \emph{cointersection Condition}:\vspace{-5pt}
\[
(u,v) \in \E \Longleftrightarrow A_u \cap A_v \neq \varnothing \text{ and } B_u \cap B_v \neq \varnothing. \vspace{-3pt}
\]
Let
$\ttc(\G) = \min_{\cR} (|\A| + |\B|)$,
where the minimum is taken over all cointersection representations $\cR$ of $\G$. 
Then $\tcG$ is called the \emph{cointersection number} of $\G$.
A cointersection representation that uses exactly $\tcG$ features is called \emph{optimal}.   
\end{definition}
It is clear 
that the cointersection number of a graph
is precisely the smallest number of features used to describe the network in the Boolean AND model (see Section~\ref{sec:Boolean}).

Fig.~\ref{fig:illustration} depicts a $(2 \mid 2)$-CIR. 
We can verify easily that for this graph, $\ttc = 4$, and hence, this representation is optimal.
If we refer to the set of nodes that have a particular common feature as a \emph{community}, 
then the community structure induced by this representation is illustrated in Fig.~\ref{fig:community}.
Note that in this setting communities are no longer restricted to be cliques, which is a more realistic modeling assumption. Furthermore, $u$ and $v$ are adjacent if and only if they belong to the \emph{intersection} of one community of type $\A$
and another community of type $\B$. Note that communities may also be defined by pairs of features, in which case they
form cliques and represent intersections of individual feature communities. 
\vspace{-5pt} 
\begin{figure}[H]
\centering
\includegraphics[scale=1]{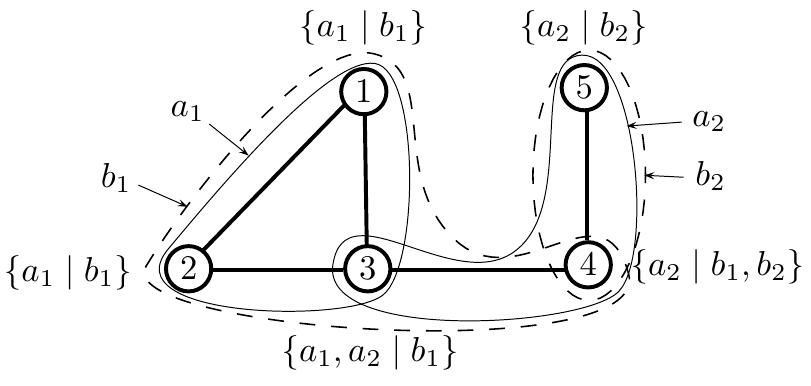}\vspace{-5pt}
\caption{The community structure induced by the features in a cointersection representation of the graph.
The vertices are grouped into different communities, each of which corresponds to an $\A$-feature
(solid closed curve) or a $\B$-feature (dashed closed curve).
The pair $(u,v)$ is an edge if and only if both $u$ and $v$ belong to a common $\A$-community and 
a common $\B$-community. In other words, every edge lies inside both a solid curve and a dashed curve.}
\label{fig:community} 
\end{figure}

In the next subsection, we review the concepts and some well-known results on the intersection
number and its generalization, the $p$-intersection number. 

\subsection{The Intersection Number and the $p$-Intersection Number}

Clearly, an $(\al \mid 1)$-CIR of a graph
is equivalent to an \emph{intersection representation} of the same graph that uses $\al$ features~\cite{ErdosGoodmanPosa1966}. An intersection representation of a graph is equivalent to an \emph{edge clique cover}, i.e. a set of complete subgraphs (cliques) of a graph that covers every edge at least once. The \emph{intersection number} of a graph $\G$, denoted
by $\toG$, is the smallest number of features used in an intersection representation 
of the graph, or the size of a smallest edge clique cover of that graph.
The \emph{$p$-intersection number} of a graph, denoted by $\theta_p(\G)$, 
is the smallest possible number of features to assign to the vertices such that two vertices are adjacent if and only if they share at least
$p$ common features (see, e.g.~\cite{Hefner1991, ChungWest1994, EatonGouldRodl1996}).
We list below a couple of well-known results on the intersection number and the $p$-intersection
number of a graph. 
\begin{theorem}[Erd\"{o}s, Goodman, and P\'{o}sa~\cite{ErdosGoodmanPosa1966}]
If $\G$ is any graph, then $\toG \leq \lfloor n^2/4 \rfloor$. 
\end{theorem}
\begin{theorem}[Alon~\cite{Alon1986}]
Let $\H$ be a graph on $n$ vertices with maximal degree at most $d$ and minimal degree at least
one, and let $\G = \overline{\H}$ be its complement. Then $\toG \leq 2e^2(d+1)^2\log_e n$. 
\end{theorem}
\begin{theorem}[Eaton, Gould, and R\"{o}dl~\cite{EatonGouldRodl1996}]
For $p \geq 2$ and any graph $\G$ on $n$ vertices, $\binom{\theta_p(\G)}{p}
\geq \toG$. 
\end{theorem}
\begin{theorem}[Eaton, Gould, and R\"{o}dl~\cite{EatonGouldRodl1996}]
Let $\G$ be a graph on $n$ vertices with maximum vertex degree $d$ and $p > 1$
be an integer, then $\theta_p(\G) \leq 3epd^2(d+1)^{1/p}n^{1/p}$. 
\end{theorem}

\section{Lower and Upper Bounds on the cointersection Numbers of Graphs}
\label{sec:upper_bounds}

We now turn our attention to deriving upper bounds on the cointersection numbers $\ttc$ of arbitrary graphs, and explicit bounds on $\ttc$ for bipartite graphs, chordal graphs, and graphs with bounded vertex degrees. 
\begin{lemma}
\label{lem:upper_bound}
For any graph $\G$, one has $\tcG \leq 1 + \toG$.
\end{lemma}
\begin{proof}
Given an optimal intersection representation of $\G$, which uses 
$\tto$ features, we may create a $(\tto \mid 1)$-CIR of $\G$ as follows. 
If in the intersection representation of $\G$ the vertex $v$ is assigned
the set of features $\{a_1,\ldots, a_r\}$, then in the corresponding 
cointersection representation of $\G$, we assign to $v$ the 
sets of features $\{a_1,\ldots,a_r \mid b\}$, where $b \notin \{a_1, \ldots, a_{\toG}\}$. 
It is easy to verify that this feature assignment is indeed a $(\tto \mid 1)$-CIR of $\G$. 
\end{proof} 

Lemma~\ref{lem:upper_bound} immediately
implies some explicit upper bounds on the cointersection number of graphs. 
For instance, the following upper bound for \emph{complement of a sparse graph} 
is an obvious corollary of 
Lemma~\ref{lem:upper_bound} and~\cite[Theorem 1.4]{Alon1986}:  
if $\G$ is a graph on $n$ vertices with maximum degree at most $n-1$ and
minimum degree at least $n-d$ then 
$\tcG \leq 1 + 2e^2(d+1)^2 \ln n$.
Another immediate consequence of Lemma~\ref{lem:upper_bound} and \cite[Corollary 3.2]{ErdosOrdmanZalcstein1993} is that if $\G$ is a chordal graph on $n$ vertices with largest
clique of size $r$ then $\tcG \leq 1+\toG \leq n - r + 2$. 

We show next that a graph of bounded degree has a cointersection representation that uses $\O(\sqrt{n})$ features. Our probabilistic proof is based on the analysis in~\cite[Theorem 11]{EatonGouldRodl1996}.   

\begin{theorem} 
\label{thm:bounded_degree}
Let $\G$ be a graph on $n$ vertices, with edge set $\mathcal{E}$ and maximum vertex degree $\Delta(\G) \leq d$. Then 
$\tcG \leq 16d^{5/2}\sqrt{n}$.
\end{theorem} 
\begin{proof}
Let $\A$ and $\B$ be two disjoint sets of features of the same cardinality
$\al = \be = 8d^{5/2}n^{1/2}$. Our goal is to show the existence of an
$(\al\mid \be)$-CIR of $\G$. 

We independently assign to
every edge $e$ of $\G$ a randomly chosen pair of features $\{a(e)\mid b(e)\}$, 
where $a(e) \in \A$ and $b(e) \in \B$. 
For each vertex $v \in \V$, let \vspace{-3pt}
\begin{equation}
\label{eq:Av}
A_v = \{a(e) \colon e=(u,v) \in \E\}, \vspace{-3pt}
\end{equation} 
\begin{equation} 
\label{eq:Bv} 
B_v = \{b(e) \colon e=(u,v) \in \E\}. \vspace{-3pt}
\end{equation} 
We aim to show that with a positive probability, the feature assignment
$\{(A_v\mid B_v) \colon v \in \V\}$ co-represents $\G$.  
Clearly, if $e=(u,v) \in \E$ then by \eqref{eq:Av} and \eqref{eq:Bv}, we have 
$a(e) \in A_u \cap A_v$ and $b(e) \in B_u \cap B_v$. 
Therefore, $A_u \cap A_v \neq \varnothing$ and $B_u \cap B_v \neq \varnothing$. 
In order for the cointersection Condition to be satisfied, we need to show that 
with a positive probability, for every $(u,v) \notin \E$, either $A_u \cap A_v = \varnothing$ or $B_u \cap B_v = \varnothing$. To this end, we make use of the Lov\'{a}sz Local 
Lemma~\cite{ErdosLovasz1975}. 

The classical Lov\'{a}sz Local Lemma may be stated as follows. Suppose that there are $m$ \emph{bad} events
$E_1,E_2,\ldots,E_m$, each occurring with probability at most $P$. Moreover, each event
is dependent on at most $D$ other events. If $PD \leq 1/4$ then 
\[
\pr(\cap_{i=1}^m \overline{E_i}) > 0.
\] 
In other words, with a positive probability, we can avoid all \emph{bad} events simultaneously. 

We define our set of \emph{bad} events as follows. 
For each $(u,v) \notin \E$, we let $E_{u,v}$ denote the event that 
$A_u \cap A_v \neq \varnothing$ and $B_u \cap B_v \neq \varnothing$. 
For each event $E_{u,v}$, we need to find an upper bounds on the probability that it happens and 
the number of other events that it may depend on.

First, we estimate the probability that each $E_{u,v}$ occurs. 
Since $\Delta(\G) \leq d$, each vertex $v \in \V$ is incident to at most $d$ edges. 
Therefore, by \eqref{eq:Av} and \eqref{eq:Bv}, $|A_v| \leq d$ and $|B_v| \leq d$, for every
$v \in \V$.
To obtain an upper bound on the probability that $A_u \cap A_v \neq \varnothing$, 
we may assume that $|A_u|$ and $|A_v|$ are as large as possible, i.e. $|A_u| = |A_v| = d$. 
Moreover, since $u$ and $v$ do not have any incident edges in common, their sets of $\A$-features
are independent. Therefore, we can treat $A_u$ and $A_v$ as two 
arbitrary subsets of $[\al]$ of sizes $d$. 
Then we have 
\[
\pr(A_u \cap A_v \neq \varnothing) \leq \dfrac{d\binom{\al}{d-1}}{\binom{\al}{d}}
= \dfrac{d^2}{\al-d+1}.  
\] 
Similarly,
\[
\pr(B_u \cap B_v \neq \varnothing) \leq \dfrac{d\binom{\al}{d-1}}{\binom{\al}{d}}
= \dfrac{d^2}{\be-d+1}.  
\] 
Thus, we deduce that for $(u,v) \notin \E$,
\begin{equation} 
\label{eq:P}
\begin{split}
\pr(E_{u,v})&=\pr(A_u \cap A_v \neq \varnothing)\times \pr(B_u \cap B_v \neq \varnothing)\\
&\leq P = \dfrac{d^4}{(\al-d+1)(\be-d+1)}.
\end{split} 
\end{equation} 

Second, we evaluate the number of other events that a certain event $E_{u,v}$ is dependent of. 
If $(u,v)\notin \E$ and $(w,x) \notin \E$ then the two events $E_{u,v}$ and
$E_{w,x}$ are dependent if and only if either there exist $z \in \{u,v\}$ and 
$z' \in \{w,x\}$ such that $(z,z') \in \E$ or $|\{u,v,w,x\}| \leq 3$.
For each $(u,v) \notin \E$, there are at most $2dn$ pairs $\{w,x\}$ that meet 
the first criteria and at most $2n$ pairs that meet the second. 
Therefore, each event $E_{u,v}$ is dependent of at most $D = 2n(d+1)$ other events. 

By Lov\'{a}s Local Lemma, it remains to prove that $PD \leq 1/4$.  
Recall that we assumed that $\al=\be=8d^{5/2}n^{1/2}$. 
Hence, we need to show that 
\begin{equation} 
\label{eq:1}
(8d^{5/2}n^{1/2}-d+1)^2 \geq 8d^4(d+1)n.
\end{equation} 
This claim may be established as follows:
\[
\begin{split}
&(8d^{5/2}n^{1/2}-d+1)^2 \geq (8d^{5/2}n^{1/2}-2\sqrt{2}d)^2\\
&= 8d^2(2\sqrt{2}d^{3/2}n^{1/2}-1)^2 = 8d^2(8d^3n - 4\sqrt{2}d^{3/2}n^{1/2} + 1)\\
&\geq 8d^2\big((d^3n + d^2n) + (7d^3n - d^2n - 4\sqrt{2}d^{3/2}n^{1/2})\big)\\
&= 8d^2\Big(d^2(d+1)n + \big((7d-1)d^{1/2}n^{1/2} - 4\sqrt{2}\big)d^{3/2}n^{1/2}\Big)\\
&> 8d^4(d+1)n.
\end{split}
\]
The last inequality is due to the fact that for $n \geq d \geq 1$, we have
$(7d-1)d^{1/2}n^{1/2} \geq 6 > 4\sqrt{2}$. 
This completes the proof.
\end{proof} 

For triangle-free $d$-regular graphs $\G$ on $n$ vertices, 
by Corollary~\ref{cr:lower_bound}, $\tcG \geq 2\sqrt{\toG} = \sqrt{2d}\sqrt{n}$. 
Therefore, in this case, the upper bound given by Theorem~\ref{thm:bounded_degree} is optimal up to a constant factor depending on $d$.
 
Recall that $\ttG$ denotes the $2$-intersection number of $\G$. 
As already pointed out, Eaton {\et}~\cite{EatonGouldRodl1996} showed that $\ttG \leq 1+\toG$ for a general graph and $\ttG \leq 3epd^2(d+1)^{1/2}\sqrt{n}$ for a graph of bounded degree $d$. 
The former bound is the same as the upper bound for $\tcG$ in Lemma~\ref{lem:upper_bound} 
and the latter is essentially the same as the upper bound for $\tcG$ in 
Theorem~\ref{thm:bounded_degree}. However, $\tcG$ and $\ttG$ can be vastly different
for certain families of graphs. For instance, we establish in Proposition~\ref{pr:Knn} 
{in Section~\ref{sec:tightness}}
that for a complete balanced bipartite graph with edge set $\V$, while $\tcG = |\V|$, 
$\ttG$ is quadratic in $|\V|$ (see Chung and West~\cite{ChungWest1994} for the latter claim).   

Next, we show that the cointersection number of a bipartite graph is at most 
its order. Since the intersection representation of a bipartite graph is equal to 
its size, the bound stated in Lemma~\ref{lem:upper_bound_bipartite}
improves the bound stated in Lemma~\ref{lem:upper_bound} when the graph
has more edges than vertices.  

\begin{lemma}
\label{lem:upper_bound_bipartite}
$\tcG \leq |\V|$ if $\G = (\V,\E)$ is a bipartite graph.
\end{lemma}
\begin{proof} 
As $\G$ is a bipartite graph, we can partition the set of vertices into 
two parts, say $U = \{1,2,\ldots,n_1\}$ and $V = \{n_1+1,n_1+2,\ldots,n\}$, 
for some $1 \leq n_1 < n$, so that $\E \subseteq \{(u,v) \colon u \in U, 
v \in V\}$. Set $\A = \{a_u\colon u \in U\}$ and $\B = \{b_v \colon v \in V\}$. We assign to each $u \in U$ 
two sets of features, namely $A_u = \{a_u\}$ and 
$B_u = \{b_v \colon (u,v) \in \E\}$.
Similarly, we assign to each $v \in V$ two sets of features, namely
$A_v = \{a_u \colon (u,v) \in \E\}$ and $B_v = \{b_v\}$. 
Then it is straightforward to verify that $\cR = \{(A_v,B_v) \colon v \in \V\}$
is an $(n_1,n-n_1)$-CIR of $\G$. As this cointersection representation  
uses $n$ features in total, the proof follows.  
\end{proof} 


We prove next a lower bound on $\ttc$ via $\tto$. 

\begin{lemma}
\label{lem:lower_bound}
If $\cR$ is an $(\al\mid \be)$-CIR of $\G$ then $\al\be \geq \tto(\G)$. As a consequence, 
$\tcG \geq \min_{\al\be \geq \toG} (\al+\be)$.
\end{lemma}
\begin{proof}
Suppose we have a cointersection representation $\cR = \{(A_v \mid B_v): v \in \V\}$  
of $\G$ with two disjoint sets of features $\A$ and $\B$, where $|\A| = \al$ and $|\B| =\be$. 
For each pair $(a,b) \in \A \times \B$, the set of vertices
$\C_{a,b} = \{v \in V: a \in A_v, b \in B_v\}$
forms a clique of $\G$. Moreover, it is obvious that any edge of $\G$
must be covered by one such clique. Therefore, 
$\cC = \{\C_{a,b}: (a,b) \in \A\times \B\}$
is an edge clique cover of $\G$. As $\toG$ is the number of cliques
in a minimum edge clique cover of $\G$, we have \vspace{-3pt}
\[
\al\be = \AB = |\cC| \geq \toG. \vspace{-3pt}
\]
Therefore,
$\tcG \geq \min_{\al\be \geq \toG} (\al + \be)$. \qedhere 
\end{proof}

The following is immediate from Lemma~\ref{lem:upper_bound}
and Lemma~\ref{lem:lower_bound}. 

\begin{corollary}
\label{cr:lower_bound} 
For any graph $\G$ we have 
\begin{equation} 
\label{eq:sandwiched}
\lceil2\sqrt{\toG}\rceil \leq \tcG \leq 1+\toG. \vspace{-3pt}
\end{equation}  
\end{corollary}

Note again that both $\ttc$ and $\ttt$ (the $2$-intersection number) have quite similar lower bounds in terms of $\tto$. 
Indeed, based on the aforementioned bound $\binom{\ttG}{2} \geq \tto(\G)$, one arrives at $\ttG \geq \sqrt{2\tto(\G)}$. 
Corollary~\ref{cr:lower_bound} gives us 
$\tcG \geq 2\sqrt{\tto(\G)}$. The two lower bounds for $\ttt$ and $\ttc$
differ from each other only by a multiplicative factor of $\sqrt{2}$. 

\section{Tightness of the Bounds}
\label{sec:tightness}

We discuss next the tightness of the bounds on $\tcG$ for several families of graphs. In addition, we link the existence of
cointersection representations of certain complete multipartite graphs that achieve the lower bound with the existence of specific 
affine planes.
 
The first result shows that for graphs with very small $\tto$, the upper bound $\tcG \leq 1 + \toG$
is actually tight. \vspace{-3pt}
\begin{proposition} 
The upper bound $\tcG \leq 1 + \toG$ stated in Lemma~\ref{lem:upper_bound}
is tight when $\toG \leq 3$. \vspace{-3pt}
\end{proposition} 
\begin{proof} 
It is obvious that when $\toG \leq 3$, the left-hand side and the right-hand
side of \eqref{eq:sandwiched} are coincide. 
\end{proof} 
Next, we demonstrate that for some simple graphs, the lower bound $\al\be \geq \tto(\G)$
established in Lemma~\ref{lem:lower_bound}
is also sufficient for the existence of an $(\al \mid \be)$-CIR. 
As $\tto$ is known for these graphs, $\ttc$ can be determined explicitly. \vspace{-3pt}
\begin{proposition}
\label{pr:simple_graphs}
If $\al\be \geq \tto(\G)$ then there exists an $(\al \mid \be)$-CIR of $\G$
when $\G$ is a star $\S_n$, a path $\P_n$, or a cycle $\C_n$.
\end{proposition}
\begin{proof}
Suppose that $\G \equiv \S_n$ is a star graph on $n$ vertices. Let $\A$ and $\B$ be two disjoint subsets of features of sizes $\al$ and $\be$, respectively. 
First, suppose that $\S_n$ has edges $(1,2),(1,3),\ldots, (1,n)$. 
Since $|\A||\B|\geq n-1=\tto(\S_n)$, we can assign distinct pairs $(a,b) \in \A\times \B$
to the edges of $\S_n$. For each vertex $v \in \{2,\ldots, n\}$, let
$A_v = \{a_{1,v}\}$, $B_v = \{b_{1,v}\}$, where $\{a_{1,v}\mid b_{1,v}\}$ are the features
assigned to the edge $(1,v)$. Also, let $A_1 = \A$ and $B_1 = \B$. It is clear that this is an $(\al \mid \be)$-CIR of $\S_n$. 

Next, suppose that $\G \equiv \P_n$ is a path on $n$ vertices and that 
it has edges $(v,v+1)$, $1\leq v < n$. 
Recall that $\tto(\P_n) = n - 1$. 
To simplify the notation, we assume that $\al\be = \tto(\P_n)=n-1$. The case when we have
strict inequality can be proved in the same manner. Furthermore, let $\A = \{a_1,\ldots,a_\al\}$, and $\B = \{b_1,\ldots,b_\be\}$.

We describe next an $(\al \mid \be)$-CIR of $\P_n$. We first split $n-1$ edges
of $\P_n$ into $\al$ equal-sized groups, each consisting of precisely $\be$ consecutive
edges. We then assign $\{a_1\mid b_1\},\{a_1\mid b_2\},\ldots,\{a_1\mid b_\be\}$
as features to the first group of $\be$ edges in that order. For the next group of $\be$ edges, 
we assign the sequence of features $\{a_2\mid b_\be\},\{a_2\mid b_{\be-1}\},\ldots,\{a_2\mid b_1\}$. 
For the third group of $\be$ edges, we use the sequence
$\{a_3\mid b_1\},\{a_3\mid b_2\},\ldots,\{a_3\mid b_\be\}$. 
Note that we used an \emph{increasing} order for the indices of the sequence $b_j$ in the first group, and a
 \emph{decreasing} order for the second group, and again an \emph{increasing} order for the
third group. We continue to assign features in this way until reaching the last group of edges.
We illustrate this feature assignment for the edges of $\P_{13}$ in the figure below.
Here, we set $\A = \{1,2,3\}$ and $\B = \{4,5,6,7\}$.  
\vspace{-15pt}
\begin{figure}[H]
\centering
\includegraphics[scale=0.7]{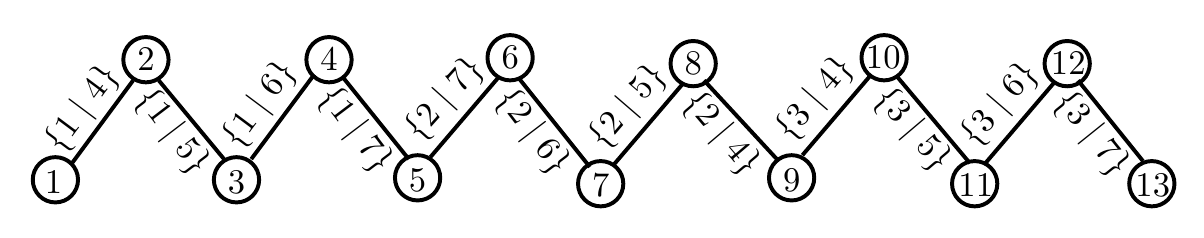}
\end{figure}
\vspace{-15pt}
We use $\{a(e)\, | \, b(e)\}$ to denote the pair of features assigned to an edge $e$. 
Then we assign to each vertex $v \in \P_n$ two feature sets $A_v = \{a(e) \colon e \text{ is 
incident to } v\}$ and $B_v = \{b(e) \colon e \text{ is 
incident to } v\}$. 
For example, the features of the vertices of $\P_{13}$ are given in the figure below. \vspace{-10pt}
\begin{figure}[H]
\centering
\includegraphics[scale=0.7]{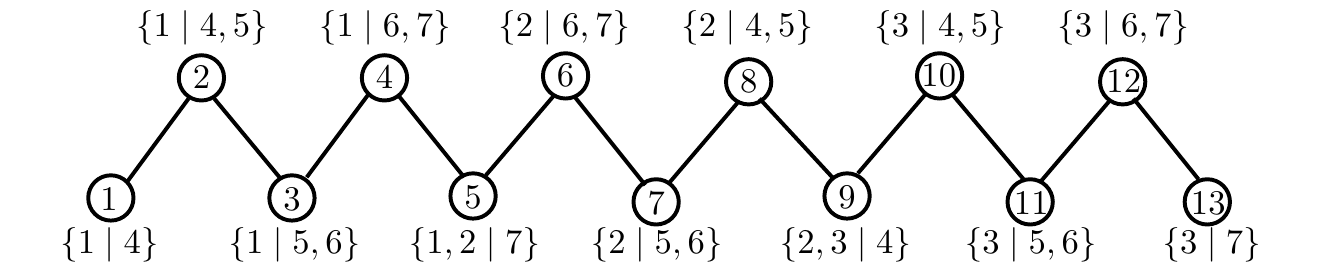}
\end{figure}
\vspace{-10pt}
We can verify that this is an $(\al \mid \be)$-CIR of $\P_n$.
Due to the way we assign features to the vertices, each vertex has precisely the feature pairs $\{a,b\}$, where $a \in \A$ and $b \in \B$ assigned to the edges incident to that vertex. 
Moreover, different edges are assigned different feature pairs. 
Consequently, two distinct vertices share a common feature pair only if they share a common edge.

\vspace{-10pt}
\begin{figure}[H]
\centering
\includegraphics[scale=0.7]{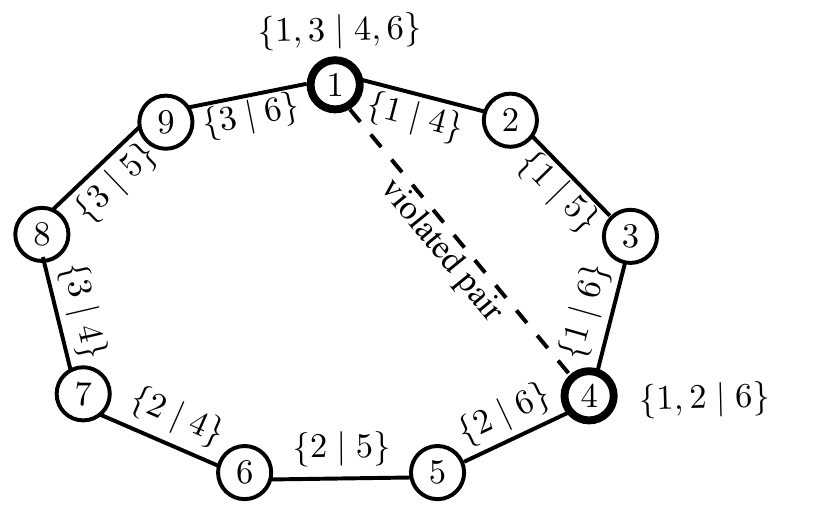}
\caption{An example where the discussed feature assignment for paths does not apply for the case of a cycle, say $\C_9$. 
Two vertices $1$ and $4$ share a pair of common features
$\{1 \mid 6\}$, even though they are not adjacent. Here we set $\A = \{1,2,3\}$ and
$\B = \{4,5,6\}$.}    
\label{fig:C9_bad} 
\end{figure}
The proof for cycles proceeds along the same lines as the proof for paths, except for one added modification. 
Recall that $\toG = n$ if $\G \equiv \C_n$ is a cycle on $n$ vertices. 
Suppose that $\al\be = n$ (the case $\al\be > n$ can be dealt with in the same manner). 
We split the $n$ edges of $\C_n$ into $\alpha$ equal-sized groups, each consisting of $\be$
consecutive edges. As demonstrated for paths, the key idea is to assign features to edges so that different edges receive different pairs of features and moreover, the set of the feature pairs each vertex has consists precisely of the feature pairs assigned to its two adjacent edges. 
When $\alpha$ is even, we assign features to $\alpha$ groups of edges of $\C_n$ 
and then deduce the set of features assigned to each vertex in the same way we do for paths. 
When $\alpha$ is odd, this feature assignment may no longer work, because now the vertex $1$ of the cycle would be assigned two sets of features $\A_1 = \{a_1, a_\al \}$ and 
$\B_1 = \{b_1, b_\beta\}$; as a result, it would have four instead of two feature pairs, namely $\{a_1\mid b_1\}$, 
$\{a_1\mid b_\be\}$, $\{a_\al\mid b_1\}$, $\{a_\al\mid b_\be\}$. 
As a consequence, this vertex may share a common pair of features with some other vertices
that are not adjacent to it. For instance, for $n = 9 = 3 \times 3$, the currently discussed feature assignment for $\C_9$, demonstrated in Fig.~\ref{fig:C9_bad}, violates the 
cointersection Condition. 
\vspace{-10pt}
\begin{figure}[htb]
\centering
\includegraphics[scale=0.7]{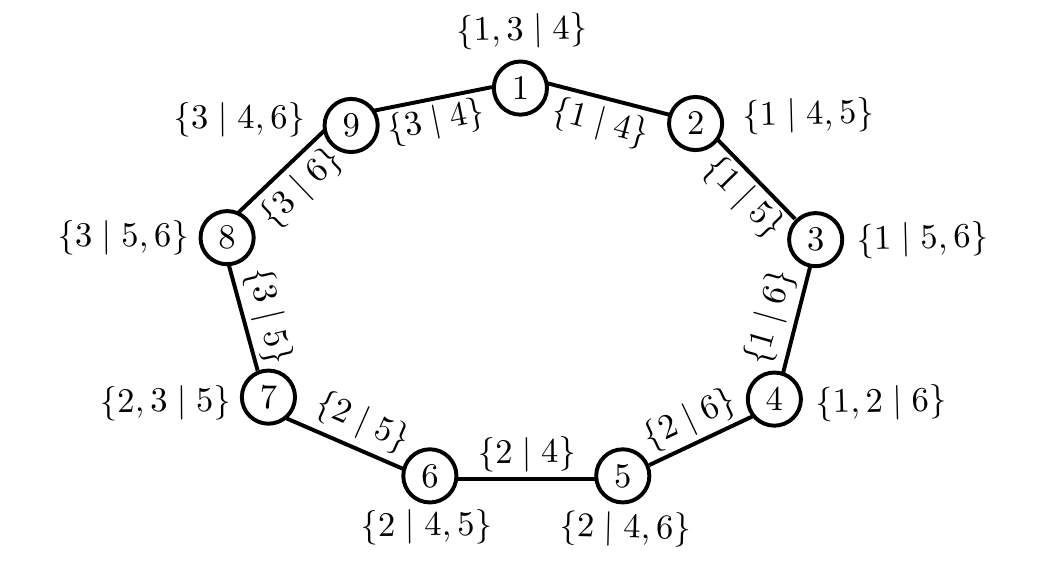}
\caption{An example of a $(3 \mid 3)$-cointersection representation of $\C_9$.
Here we set $\A = \{1,2,3\}$ and $\B = \{4,5,6\}$.}  
\label{fig:C9_good} 
\end{figure}

We correct this issue as follows.
Suppose that $\al \geq 3$ (the case $\al = 1$ and $\be = n$ is trivial, due to Lemma~\ref{lem:upper_bound}). We assign features to the first $\al-2$ groups of edges of $\C_n$
in the same way as for paths. For the $(\al-1)$th group, instead of assigning 
$\{a_{\al-1} \mid b_\be\}, \ldots, \{a_{\al-1} \mid b_1\}$, 
we assign $\{a_{\al-1} \mid b_\be\}, \ldots, \{a_{\al-1} \mid b_3\}, \{a_{\al-1} \mid b_1\},
\{a_{\al-1} \mid b_2\}$ to the edges in this order. 
For the $\al$th group, instead of assigning $\{a_{\al} \mid b_1\}, \ldots, \{a_{\al} \mid b_\be\}$, we assign $\{a_{\al} \mid b_2\}, \{a_{\al} \mid b_3\},\ldots, \{a_{\al} \mid b_\be\},
\{a_{\al} \mid b_1\}$ to the edges. In this way, we guarantee that the vertex $1$
is also assigned two feature pairs as the others, and hence, two vertices share a common
feature pair if and only if they are adjacent to the same edge. 
We illustrate this feature assignment in Fig.~\ref{fig:C9_good}. 
\end{proof} 

\begin{corollary}
\label{cr:simple_graphs}
If $\G$ is a star, a path, or a cycle, then
$\lceil 2\sqrt{\toG}\rceil \leq \tcG \leq 2\lceil\sqrt{\toG}\rceil$.  
\end{corollary}
\begin{proof}
By Corollary~\ref{cr:lower_bound}, we have $\tcG \geq \lceil 2\sqrt{\toG}\rceil$. 
Moreover, by Proposition~\ref{pr:simple_graphs}, if $\G$ is a star, a path, 
or a cycle, then there exists a
$(\lceil\sqrt{\toG}\rceil \mid \lceil\sqrt{\toG}\rceil)$-CIR of $\G$, 
which uses $2\lceil\sqrt{\toG}\rceil$ features in total. Hence, 
$\lceil 2\sqrt{\toG}\rceil \leq \tcG \leq 2\lceil\sqrt{\toG}\rceil$, 
which establishes our assertion for stars, paths, and cycles.
\end{proof}

Similar results also hold for complete multipartite graphs $\K_{n,\ldots,n}$ with certain parameters, 
as shown in the subsequent results. 
Note that for a complete bipartite graph $\K_{n,n}$, we have
$\tto(\K_{n,n}) = n^2$, which is precisely the number of edges. 
We henceforth denote the set $\{1,2,\ldots, m\}$ by $[m]$. 

\begin{proposition}
\label{pr:Knn}
If $n = ts$ then a $(t,ts^2)$-CIR exists for $\K_{n,n}$. As a consequence, 
$\ttc(\K_{n,n}) = 2n = 2\sqrt{\tto(\K_{n,n})}$. 
\end{proposition}
\begin{proof} 
The explanation that the second assertion follows from the first assertion is 
as follows. Let $t = n$ and $s = 1$. Then an $(n,n)$-CIR 
of $\K_{n,n}$ exists which uses exactly $2n$ features. Combining this result with Corollary~\ref{cr:lower_bound}, we have
\[
2n = 2\sqrt{\tto(\K_{n,n})} \leq \ttc(\K_{n,n}) \leq 2n,
\] 
which implies that 
\[
\ttc(\K_{n,n}) = 2n = 2\sqrt{\tto(\K_{n,n})}. 
\]
Note that this equality may also be deduced by combining Corollary~\ref{cr:lower_bound}
and Lemma~\ref{lem:upper_bound_bipartite}. 

We now prove the first assertion of the proposition. 
Let $\A = \{a_1,\ldots,a_t\}$ and $\B = \{b_1,\ldots,b_{ts^2}\}$. 
Let $R_1,\ldots,R_s$ be disjoint subsets of size $ts$ of $\B$ that partition $\B$.
Moreover, let $C_1,\ldots,C_{ts}$ be disjoint subsets of size $s$ of $\B$ that 
partition $\B$. In addition, let $|R_i \cap C_j| = 1$ for every $i \in [s]$ and 
$j \in [ts]$. For instance, if we arrange the $ts^2$ elements of $\B$ in a $s\times(ts)$
matrix, then we can simply let $R_i$ be the set of $ts$ elements in the $i$th row
and let $C_j$ be the set of $s$ elements in the $j$th column. 

We assign feature sets to each vertex in $\K_{n,n}$ as follows. 
Suppose that $\V(\K_{n,n}) = \{1,\ldots,n\} \cup \{n+1,\ldots,2n\}$, and let 
$\E(\K_{n,n}) = \{(i,j) \colon 1 \leq i \leq n, n+1 \leq j \leq 2n\}$. 
First, for a vertex $i \in \{1,\ldots,n\}$, we write $i = (i_a-1) s + i_b-1$, where
$1 \leq i_a \leq t$ and $1 \leq i_b \leq s$.
Then we assign $A_i = \{a_{i_a}\}$ and $B_i = R_{i_b}$. 
For a vertex $i \in \{n+1,\ldots,2n\}$, we assign $A_i = \A = \{a_1,\ldots,a_t\}$ and 
$B_i = C_i$. Recall that $n = ts$, which is precisely the number of sets
$C_j$'s that we have. For example, when $n = 6$, $t = 2$, and $s = 3$, 
then the sets $R_i$ and $C_j$ consist of elements in the correspondingly indexed 
rows and columns, respectively, of the matrix given below.  
\vspace{-5pt}
\begin{figure}[H]
\centering
\begin{tabular}{ r|c|c|c|c|c|c| }
\multicolumn{1}{r}{}
 & \multicolumn{1}{c}{$C_1$} & \multicolumn{1}{c}{$C_2$} 
 & \multicolumn{1}{c}{$C_3$} & \multicolumn{1}{c}{$C_4$} 
 & \multicolumn{1}{c}{$C_5$} & \multicolumn{1}{c}{$C_6$}\\
\cline{2-7}
$R_1$ & $b_1$ & $b_2$ & $b_3$ & $b_4$ & $b_5$ & $b_6$\\
\cline{2-7}
$R_2$ & $b_7$ & $b_8$ & $b_9$ & $b_{10}$ & $b_{11}$ & $b_{12}$\\
\cline{2-7}
$R_3$ & $b_{13}$ & $b_{14}$ & $b_{15}$ & $b_{16}$ & $b_{17}$ & $b_{18}$\\
\cline{2-7}
\end{tabular}
\end{figure}
The resulting $(2,18)$-CIR of $\K_{6,6}$ constructed
as described above is illustrated in Fig.~\ref{fig:K66}.
\vspace{-5pt}
\begin{figure}[H]
\centering
\includegraphics[scale=0.7]{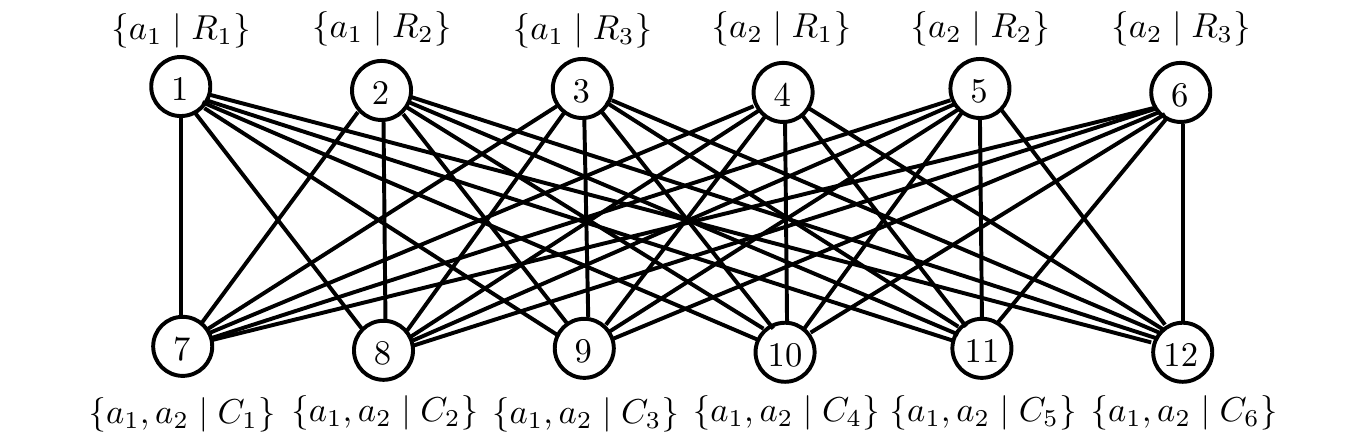}
\caption{A $(2,18)$-CIR of $\K_{6,6}$.
The sets $R_1$, $R_2$, and $R_3$ are pairwise disjoint. 
The sets $C_1,\ldots,C_6$ are also pairwise disjoint.
Each pair of sets $R_i$ and $C_j$ has an intersection of size one.
Both $R_i$'s and $C_j$'s are subsets of $[b_1,\ldots,b_{18}]$.}
\label{fig:K66}
\end{figure}

We now proceed to verify that this feature assignment is indeed a cointersection representation of $\K_{n,n}$.   

We first verify that the cointersection Condition holds for non-edges of $\K_{n,n}$.
For $1 \leq i \neq i' \leq n$, either $i_a \neq i'_a$ or $i_b \neq i'_b$. 
If $i_a \neq i'_a$ then $A_i \cap A_{i'} = \{a_{i_a}\} \cap \{a_{i'_a}\} = \varnothing$. 
If $i_b \neq i'_b$ then $B_i \cap B_{i'} = R_{i_b} \cap R_{i'_b} = \varnothing$, 
because the sets $R_i$ form a partition.  
In either case, we have $A_i \cap A_{i'} = \varnothing$ or $B_i \cap B_{i'}
= \varnothing$.  
For $n+1 \leq i \neq i' \leq 2n$, we always have $B_i \cap B_{i'}
= C_i \cap C_{i'} = \varnothing$, since all the pairs of sets $C_i$ are disjoint. 

Next, we verify that the cointersection Condition holds for edges of $\K_{n,n}$.
Indeed, for $1 \leq i \leq n$ and $n+1 \leq j \leq 2n$, we have
$A_i \cap A_{j} = \{a_{i_a}\} \cap \A = \{a_{i_a}\} \neq \varnothing$, and moreover,
$B_i \cap B_{j} = R_{i_b} \cap C_j \neq \varnothing$, because we assume that
$|R_i \cap C_j| = 1$ for every $i \in [s]$ and $j \in [ts]$. Thus, we constructed a $(t,ts^2)$-CIR of $\K_{n,n}$. 
\end{proof} 

Before proceeding with our discussion, we review a few definitions from the theory of combinatorial designs 
(see, e.g.~\cite[VI.40]{ColbournDinitz2006}). Let $n \geq k \geq 2$. A $2$-$(n,k,1)$ \emph{packing} is a pair $(\X,\S)$, 
where $\X$ is a set of $n$ elements (points) and $\S$ is a collection of
subsets of size $k$ of $\X$ (blocks), such that every pair of points
occurs in \emph{at most} one block in $\S$.  
A $2$-$(n,k,1)$ packing $(\X,\S)$ is \emph{resolvable} if $\S$ can be partitioned
into \emph{parallel classes}, each comprising $n/k$ blocks
that partition $\X$. We provide an example for a $2$-$(9,3,1)$ resolvable packing below. \vspace{-5pt}
\begin{figure}[H]
\centering
\begin{tabular}{cccc}
$\{1,2,3\}$ & \qquad $\{1,4,7\}$ & \qquad $\{1,5,9\}$ & \qquad $\{1,6,8\}$\\
$\{4,5,6\}$ & \qquad $\{2,5,8\}$ & \qquad $\{2,6,7\}$ & \qquad $\{2,4,9\}$\\
$\{7,8,9\}$ & \qquad $\{3,6,9\}$ & \qquad $\{3,4,8\}$ & \qquad $\{3,5,7\}$
\end{tabular}
\caption{A $2$-$(9,3,1)$ resolvable packing with four parallel classes.}
\label{fig:affine_plane}
\vspace{-5pt}
\end{figure}

The following simple lemma describes a property of a $2$-$(k^2,k,1)$ resolvable
packing that will be of importance in the proof of upcoming Theorem~\ref{thm:packing}. 

\begin{lemma} 
\label{lem:packing}
Let $(\X,\S)$ be a $2$-$(k^2,k,1)$ resolvable packing. If $S \in \S$ and 
$S' \in \S$ are two blocks from different parallel classes, then 
$|S \cap S'| = 1$.   
\end{lemma} 
\begin{proof}
By the definition of a packing, every pair of points is contained in exactly
one block. Therefore, any two different blocks have at most one point in common.
Hence, $|S \cap S'| \leq 1$.  
Suppose that $S$ and $S'$ belong two different parallel classes $\C$
and $\C'$, respectively. Note that each parallel class consists of precisely
$k = k^2/k$ disjoint blocks. These $k$ blocks together partition the set $\X$.
Therefore, if $S' \notin \C$ then it must intersect each block in $\C$ at 
at least one point, for otherwise 
\[
|S'| = |\cup_{S \in \C} S' \cap S| = \sum_{S \in \C} |S' \cap S| < 
\sum_{S \in \C} 1 = k,
\]
a contradiction. Hence, $|S \cap S'| \geq 1$. Thus, $|S \cap S'| = 1$.    
\end{proof} 

\begin{theorem}
\label{thm:packing}
If there exists a $2$-$(k^2, k, 1)$-resolvable packing with at least $r \geq 2$ parallel classes then $\ttc(\knr) = 2n$, where $n = k^2$, and $\knr$ is the complete $r$-partite graph $\K_{n,\ldots,n}$.
\end{theorem}
\begin{proof} 
Note that for $r \geq 2$, $\knn$ is an induced subgraph of
$\knr$. Therefore, by Proposition~\ref{pr:Knn}, we have
\[
\ttc(\knr) \geq \ttc(\knn) = 2n. 
\]
Hence, it remains to prove that we can co-represent $\knr$ by using $2n$
features if a certain resolvable packing exists.
 
Let us assume that a $2$-$(n=k^2, k, 1)$-resolvable packing $(\X,\S)$ 
with at least $r$ parallel classes, say $\C_1,\ldots,\C_r$, exists. 
Let $\A = \{a_x \colon x \in \X\}$ and $\B = \{b_x \colon x \in \X\}$. 
Then $|\A| = |\B| = n$. We assign to the vertices of $\knr$ features
from $\A$ and $\B$ as follows. 
Consider $n$ vertices in the $\ell$th part $P_\ell$ of the graph $(\ell \in [r])$. 
We partition these $n=k^2$ vertices into $k$ groups, each of which consists
of precisely $k$ vertices. Let $G^\ell_i = \{v^{\ell}_{i,j} \colon j \in [k]\}$ denote
the $i$th vertex group of $P_\ell$, for $i \in [k]$ and $\ell \in [r]$. 
The vertices in $P_\ell$ are then assigned features according to the 
blocks in the $\ell$th parallel class $\C_\ell = \{S^\ell_1,\ldots,S^\ell_k\}$
in the following way. 
The vertex $v^{\ell}_{i,j}$ in the $i$th group $G^\ell_i$ has feature sets
$A_{v^{\ell}_{i,j}} = \{a_x \colon x \in S^\ell_i\}$ and $B_{v^{\ell}_{i,j}} = \{b_x \colon x \in S^\ell_j\}$. 

We show next that the above feature assignment indeed satisfies the cointersection Condition. 

First, we verify this condition for the non-edges of $\knr$. 
Consider each part $P_\ell$ of the graph. 
If $v^{\ell}_{i,j}$ and $v^{\ell}_{i,j'}$, where
$j \neq j'$, are two 
distinct vertices that belong to the same group $G^\ell_i$, then 
\[
|B_{v^{\ell}_{i,j}} \cap B_{v^{\ell}_{i,j'}}| = |S^\ell_j \cap S^\ell_{j'}| = 0.
\] 
The reason is that when $j \neq j'$, $S^\ell_j$ and $S^\ell_{j'}$ are two distinct
blocks in the same parallel class $\C_\ell$ of the packing, and hence must be
disjoint. If $v^{\ell}_{i,j}$ and $v^{\ell}_{i',j'}$ belong to different groups $G^\ell_i$
and $G^\ell_{i'}$, respectively, where $i \neq i'$, then
\[
|A_{v^{\ell}_{i,j}} \cap A_{v^{\ell}_{i',j'}}| = |S^\ell_i \cap S^\ell_{i'}| = 0,
\] 
because $S^\ell_i$ and $S^\ell_{i'}$ are two distinct blocks in the same
parallel class $\C_\ell$. Thus, every pair of vertices from the same part $P_\ell$ $(\ell \in [r])$
has either no $\A$-features or no $\B$-features in common. 

Second, we verify the cointersection Condition for the edges of $\knr$ that 
connect vertices in different parts. Suppose that $v^{\ell}_{i,j} \in P_\ell$
and $v^{\ell'}_{i',j'} \in P_{\ell'}$, where $P_\ell$ and $P_{\ell'}$ are different
parts of the complete $r$-partite graph. Then we have
\[
|A_{v^{\ell}_{i,j}} \cap A_{v^{\ell'}_{i',j'}}| = |S^\ell_i \cap S^{\ell'}_{i'}| = 1.
\]  
The validity of the above claim follows from the observation that for $\ell \neq \ell'$, the two blocks 
$S^\ell_i$ and $S^{\ell'}_{i'}$, which are from different parallel classes of the 
packing, must intersect at one point (according to Lemma~\ref{lem:packing}). 
Similarly, we have
\[
|B_{v^{\ell}_{i,j}} \cap B_{v^{\ell'}_{i',j'}}| = |S^\ell_j \cap S^{\ell'}_{j'}| = 1.
\]  
Therefore, the cointersection Condition is satisfied for all edges of the graph. 
Thus, the assigned features form an $(n,n)$-CIR of
$\knr$, which uses precisely $2n$ features, as desired. 
\end{proof}

\begin{example}
\label{ex:K9999}
To illustrate the idea of Theorem~\ref{thm:packing}, we consider
$\K_{9,9,9,9}$ and the $2$-$(9,3,1)$ resolvable packing with four parallel 
classes $\C_1,\C_2,\C_3, \C_4$ given in Fig.~\ref{fig:affine_plane}. 
Note that by Theorem~\ref{thm:packing}, \vspace{-5pt}
\[
\ttc(\K_{9,9,9,9}) = \ttc(\K_{9,9,9}) = \ttc(\K_{9,9}) = 2\sqrt{\tto(\K_{9,9})}
=18.
\]
We omit the edges of the graph and provide a $(9,9)$-CIR 
of $\K_{9,9,9,9}$ in Fig.~\ref{fig:K9999}. Note that in this figure, instead
of $a_i$ and $b_j$, we simply use $i$ and $j$, respectively.  
\vspace{-5pt}
\begin{figure}[H]
\centering
\includegraphics[scale=0.52]{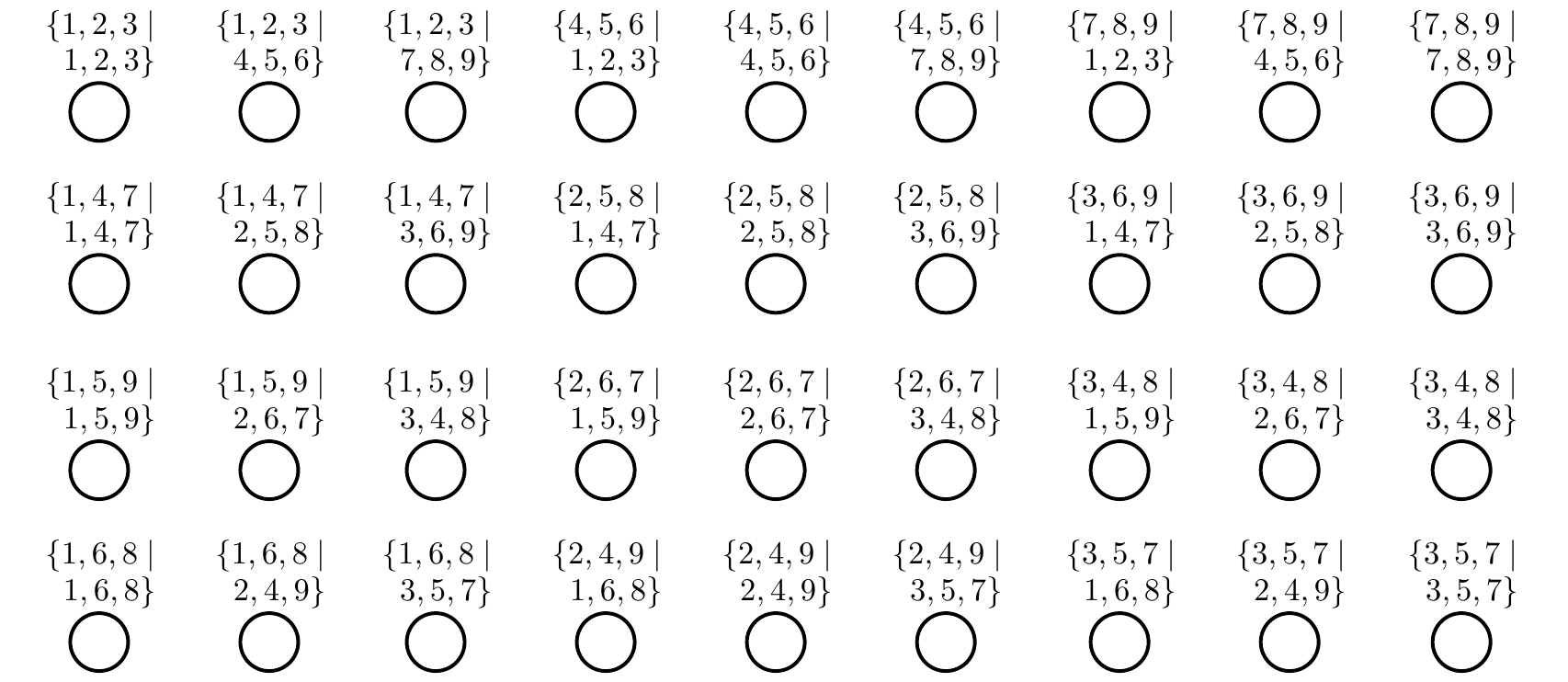}
\caption{An optimal $(9,9)$-CIR of $\K_{9,9,9,9}$
via a $2$-$(9,3,1)$ resolvable packing with four classes. In fact, this is 
a $2$-$(9,3,1)$ resolvable design, which is also an affine plane of order $9$.}
\label{fig:K9999} \vspace{-5pt}
\end{figure}
\end{example}

A $2$-$(n,k,1)$ resolvable \emph{design} (see, e.g.~~\cite[II.7]{ColbournDinitz2006}) is equivalent to a $2$-$(n,k,1)$
resolvable packing defined earlier, except that one requires that
every pair of points appear in \emph{exactly} one block. 
An \emph{affine plane} of order $k$ is a $2$-$(k^2,k,1)$
resolvable design. So far, only affine planes of orders that are
prime powers are known (see, e.g.~\cite[VII.2.2]{ColbournDinitz2006}).   

\begin{corollary}
\label{cr:affine_plane}
If there exists an affine plane of order $k$ then $\ttc(\knr) = 2n$, for every $r \leq k+1$, 
where $n = k^2$. As a consequence, this equality holds when $k$ is a prime power.  
\end{corollary}
\begin{proof} 
It is well known that a $2$-$(k^2,k,1)$ resolvable design has precisely $k+1$
parallel classes.
As an affine plane of order $k$ is a $2$-$(k^2,k,1)$ resolvable design, which
is also a packing, by Theorem~\ref{thm:packing}, the first assertion of the
corollary follows. The last assertion also holds because an affine plane of a prime
power order always exists. 
The resolvable packing used in Example~\ref{ex:K9999} is in fact an affine 
plane of order three. 
\end{proof} 

In light of Corollary~\ref{cr:affine_plane}, it is apparently nontrivial to prove
(theoretically or computationally) 
that $\ttc(\K_{n^{[r]}}) > 2n$, where $n = k^2$, $r = k+1$, 
when $k$ is not a prime power. Indeed, such a proof (if any) would imply that 
an affine plane of order $k$ does not exist. Note that the question whether 
an affine plane of an order which is not a prime power exists is still a widely open question in 
finite geometry. It is not even known whether an affine plane of order $12$
or $15$ exists (see, e.g.~\cite[VII.2.2]{ColbournDinitz2006}).    

\begin{corollary}
\label{cr:4}
$\ttc(\K_{n,n,n}) = 2n$ for every $n = k^2$, where $k \geq 2$ is not necessarily a prime power. 
\end{corollary}
\begin{proof} 
By Theorem~\ref{thm:packing}, it suffices to construct a $2$-$(k^2,k,1)$ resolvable
packing with three parallel classes for every $k \geq 2$. 
Let $\X = [k^2]$. We can arrange these $k^2$ points into a $k \times k$ matrix.
Then the $k$ blocks containing the points along the rows of this matrix form the first parallel class. 
The $k$ blocks containing the points along the columns of this matrix form the second parallel
class. 
The $k$ blocks containing the points along the direction of the main diagonal form the third parallel class. 
It is easy to verify that these blocks and the three parallel classes form a $2$-$(k^2,k,1)$ resolvable packing. 
\begin{figure}[H]
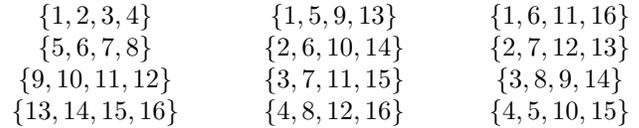

\centering
\begin{tabular}{ccc}
$\{1,2,3,4\}$ & \qquad $\{1,5,9, 13\}$ & \qquad $\{1, 6, 11, 16\}$\\
$\{5,6,7,8\}$ & \qquad $\{2,6,10,14\}$ & \qquad $\{2, 7, 12, 13\}$\\
$\{9,10,11,12\}$ & \qquad $\{3,7,11,15\}$ & \qquad $\{3,8,9,14\}$\\
$\{13,14,15,16\}$ & \qquad $\{4,8,12,16\}$ & \qquad $\{4,5,10,15\}$
\end{tabular}
\caption{A $2$-$(16,4,1)$ resolvable packing with three parallel classes.}
\label{fig:packing}
\end{figure} 
For example, when $k = 4$, the three
parallel classes of this packing are given in Fig.~\ref{fig:packing}.  
\end{proof} 

Until this point, we have focused on providing several examples of graphs which meet the lower
bound on $\ttc$ established in Lemma~\ref{lem:lower_bound}. 
However, as we establish in subsequent propositions, the lower bound many not always be achievable. 
Note that by Corollary~\ref{cr:4}, $\ttc(\K_{n,n,n}) = 2n$ for 
$n = 4, 9, 16, \ldots$ This is, in contrast, not true for $n = 2,3$. 

We first need to prove the following lemma, which states an important property of cointersection representations of triangle-free graphs (e.g. bipartite graphs) that meet the lower bound on $\ttc$ in Lemma~\ref{lem:lower_bound}. 
Recall that if $\G=(\V,\E)$ is a triangle-free graph, then $\toG = |\E|$. 

\begin{lemma} 
\label{lem:degree}
If there exists an $(\al\mid \be)$-CIR of a 
triangle-free graph $\G = (\V,\E)$ where $\al\be=|\E|$, then 
\[
|A_v||B_v| = \deg(v),
\]
for every $v \in V$. Moreover, if $(u,v) \in \E$, then $|A_u \cap A_v| = 
|B_u \cap B_v| = 1$. 
\end{lemma}
 
\begin{proof} 
Suppose that $\{(A_v,B_v) \colon v \in \V\}$ is an $(\al\mid \be)$-CIR of $\G$, where $\al\be = |\E|$. 
For each edge $(u,v) \in \E$, 
choose an arbitrary feature $a_{u,v} \in A_u \cap A_v$ and an arbitrary feature
$b_{u,v} \in B_u \cap B_v$ and assign the pair $\{a_{u,v}\mid b_{u,v}\}$
to this edge. 

We claim that different edges must have different pairs of features. 
Indeed, if $(u,v)$ and $(u',v')$ are two different edges of $\G$ such that
$a_{u,v} = a_{u',v'}$ and $b_{u,v} = b_{u',v'}$, then the four vertices
$u,v,u',v'$ have a pair of features in common, namely $\{a_{u,v} \mid b_{u,v}\}$. 
This implies that any three distinct vertices among these four must form a triangle in $\G$, which
contradicts our assumption that $\G$ is triangle-free. 
Thus, different edges must be assigned different pairs of features, as claimed.
A consequence of this claim is that for every vertex $v \in \V$, the number of 
pairs of features $\{a\mid b\}$, where $a \in A_v$ and $b \in B_v$, must be 
greater than or equal to the number of edges incident to $v$. In other words,
$|A_v||B_v| \geq \deg(v)$, for every $v \in \V$. 

Moreover, by our assumption, the number of possible pairs of features $\{a\mid b\}$, 
where $a \in \A$ and $b \in \B$, is $\al\be$, which is the same as
the number of edges. Therefore, each such pair of features must be used exactly 
once, as features of some edge. It is now clear that if $(u,v) \in \E$, then
$|A_u \cap A_v| = 1$ and $|B_u \cap B_v| = 1$. For otherwise, we could replace
the assigned features $\{a_{u,v} \mid b_{u,v}\}$ for $(u,v)$ 
by a different pair of features $\{a'\mid b'\}$, 
where $a' \in A_u \cap A_v$ and $b' \in B_u \cap B_v$. But as proved earlier, 
$\{a'\mid b'\}$ must already have been used as a pair of features of some other edge $(u',v') \neq (u,v)$. 
That would imply a triangle formed by some three distinct vertices among $u,v,u'$,
and $v'$, which, again, contradicts our assumption that $\G$ is triangle-free.  

Finally, suppose that $|A_v||B_v| > \deg(v)$ for some $v \in \V$. 
Then there must be a pair of features $\{a \mid b\}$, where $a \in A_v$ and 
$b \in B_v$, that is not assigned to any edge incident to $v$. 
However, as shown earlier, this pair of features $\{a \mid b\}$ must be used
as features of some edge, say $(u,w)$, that is not incident to $v$. Then
$u$, $v$, and $w$ share the common features $a\in \A$ and $b \in \B$ and hence 
must form a triangle in $\G$, which is impossible. Thus, $|A_v||B_v| = \deg(v)$
for every $v \in \V$, as stated. 
\end{proof} 

\begin{proposition}
$\ttc(\K_{n,n,n}) > 2n$ for $n = 2,3$. Hence, for the given graphs, the lower bound 
$\ttc \geq \min_{\al\be \geq \tto} (\al+\be)$ established in
Lemma~\ref{lem:lower_bound} is not tight. 
\end{proposition}
\begin{proof} 
Since the graphs under consideration are small, one can determine their cointersection numbers by using the algorithm of Section~\ref{subsec:SAT}, resulting in $\ttc(\K_{2,2,2}) = 5$ and $\ttc(\K_{3,3,3}) = 8$. This fact may also be proved theoretically, based on the previously derived results for the induced subgraphs $\K_{2,2}$ and $\K_{3,3}$. The details of the proof are omitted due to lack of space.
\end{proof} 

\begin{proposition}
\label{pr:KMnn}
Let $\K^M_{n,n}$ be a bipartite matrix obtained from $\K_{n,n}$ by removing
a maximum matching. 
Then \vspace{-3pt}
\[
2n-1 \leq \ttc(\K^M_{n,n}) \leq 2n. \vspace{-3pt}
\] 
The lower bound is attained when $n = 2,3$. 
If $n - 1$ is an odd prime, then $\ttc(\K^M_{n,n}) = 2n$.
\end{proposition}
\begin{proof} 
Let $\K^M_{n,n} = (\V,\E)$ and
let $U = \{u_1,\ldots,u_n\}$ and $V = \{v_1,\ldots,v_n\}$
be two parts of $\V$ such that
\[
\E = \{(u_i,v_j) \colon 1\leq i \neq j \leq n\}.
\] 
By Lemma~\ref{lem:upper_bound_bipartite} we have 
\begin{equation} 
\label{eq:2}
\ttc(\K^M_{n,n}) \leq 2n.
\end{equation} 
Note that 
\[
\tto(\K^M_{n,n}) = |\E(\K^M_{n,n})| = n^2 - n = (n-1)n.
\]
Therefore, by Lemma~\ref{lem:lower_bound}, 
\begin{equation} 
\label{eq:KM}
\ttc(\K^M_{n,n}) \geq \min_{\al\be \geq n(n-1)} (\al + \be) = (n-1) + n = 2n-1.
\end{equation} 
When $n = 2,3$, the above lower bound on $\ttc$ is attained. Examples of 
$(n-1,n)$-CIRs of $\K^M_{n,n}$ when $n = 2,3$ 
are given in Fig.~\ref{fig:KMnn}. 
\vspace{-10pt}
\begin{figure}[H]
\centering
\includegraphics[scale=0.7]{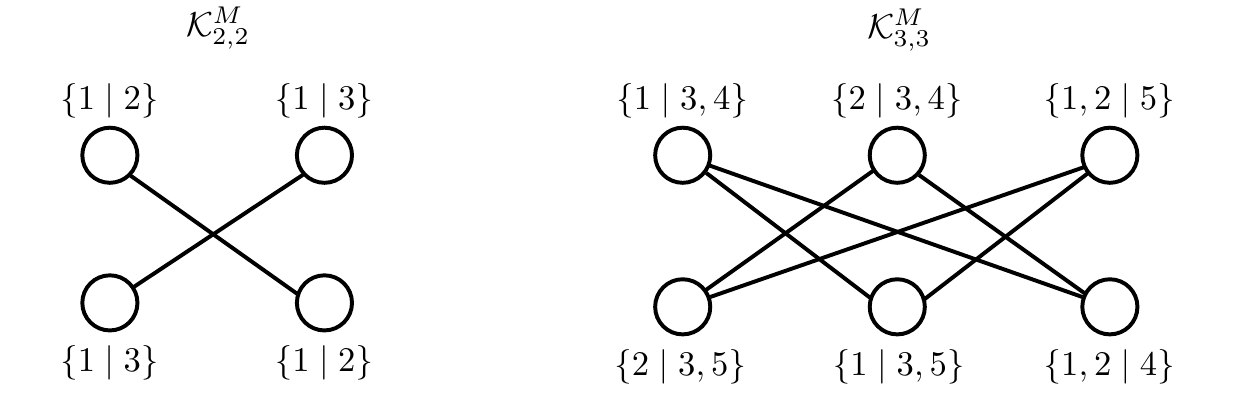}
\caption{A $(1,2)$-CIR of $\K^M_{2,2}$ (left)
and a $(2,3)$-CIR of $\K^M_{3,3}$ (right).}
\label{fig:KMnn}
\end{figure}
\vspace{-5pt}

It remains to show that if $n-1$ is an odd prime then $\ttc(\K^M_{n,n}) \neq 2n-1$.
Suppose, by contradiction, that $\ttc(\K^M_{n,n}) = 2n-1$. 
Then there must exist an $(n-1,n)$-CIR of $\K^M_{n,n}$.
Let $\A = \{a_1,\ldots,a_{n-1}\}$ and $\B = \{b_1,\ldots,b_n\}$. 
Note that every vertex of this graph has degree $n-1$. 
By Lemma~\ref{lem:degree}, for every vertex $v$, 
\[
|A_v||B_v| = \deg(v) = n - 1. 
\]
As $n-1$ is a prime number, we deduce that either $|A_v| = 1$
and $|B_v| = n - 1$ or $|A_v| = n - 1$ and $|B_v| = 1$. 
We consider the following three cases, distinguished by the number of vertices
that have only one $\A$-feature, and aim to obtain a contradiction in each
case:
\begin{itemize}
	\item \textbf{Case 1.} $|A_v| = n - 1$ and $|B_v| = 1$ for all $v \in \V = U \cup V$. 
		Since $A_{u_i} = \A$ for all $i \in [n]$ and there are no edges between
		these vertices $u_i$ elements, $B_{u_i} \cap B_{u_{j}} = \varnothing$ whenever $i \neq j$. 
		Similarly, $B_{v_i} \cap B_{v_j}$ whenever
			$i \neq j$. However, as $u_1$ is adjacent to $v_2,\ldots,v_n$, 
				these vertices must have the same $\B$-feature as $u_1$. We arrive at a contradiction. 
	\item \textbf{Case 2.} There exists one vertex, say
	$u_i$, satisfying $|A_{u_i}|$ = 1, while other vertices in the same part
		have $A_{u_j} = \A$, $j \neq i$. By Lemma~\ref{lem:degree},  
		$|B_{u_i}| = n - 1$ and $|B_{u_j}| = 1$ for $j \neq i$.
		Moreover, as $u_i$ is not adjacent to $u_j$ for $j \neq i$, 
		$B_{u_i} \cap B_{u_j} = \varnothing$. As $|\B| = n$ and $|B_{u_i}| = n - 1$, 
		this implies that $B_{u_j} = \B \setminus B_{u_i}$ for all $j \neq i$. 
		Since $A_{u_j} = \A$ for all $j \neq i$ as well, the corresponding elements $u_j$ must be 
		all adjacent, which is not true. We arrive at a contradiction.
	\item \textbf{Case 3.}  There exist two vertices, which we without loss of generality label as 
	$u_i$ and $u_j$, that are in the same part of the graph, and which satisfy 
	$|A_{u_i}| = |A_{u_j}| = 1$. Then by Lemma~\ref{lem:degree}, 
	$|B_{u_i}| = |B_{u_j}| = n - 1$. Since $n > 2$, $B_{u_i} \cap B_{u_j} \neq 
	\varnothing$. Therefore, $A_{u_i} \cap A_{u_j} = \varnothing$.
	Without loss of generality, let $A_{u_i} = \{a_i\}$ and $A_{u_j} = \{a_j\}$. 
	For any $h \in [n] \setminus \{i,j\}$, since $v_h$ is connected to 
	both $u_i$ and $u_j$, we deduce that $\{a_i,a_j\}$ is a subset of both 
	$A_{v_h}$. As $|A_{v_h}| \in \{1,n-1\}$, we deduce that $A_{v_h} = \A$, 
	for all $h \neq i,j$. Then $|B_{v_h}| = 1$ and $B_{v_h} \cap B_{v_k} =\varnothing$
	for every $h \neq k$, $h,k \in [n] \setminus \{i,j\}$. 
	We can set $B_{v_h} = \{b_h\}$ and $B_{v_k} = \{b_k\}$. 
	As $n -1 $ is an odd prime, $n \geq 4$. 
	Therefore, we can choose $h$ and $k$ such that $h, k, i,j$ are distinct.
	
	Since $v_h$ and $v_k$ are not adjacent to $v_i$, 
	and moreover, since $A_{v_h} = A_{v_k} = \A$, we deduce that 
	$B_{v_i} \cap \{b_h,b_k\} = \varnothing$. 
	Therefore, $|B_{v_i}| \leq n - 2$. Since $|B_{v_h}| \in \{1,n-1\}$, we deduce
	that $|B_{v_i}| = 1$. Similarly, $|B_{v_j}| = 1$. 
	We can set $B_{v_i} = \{b_i\}$ and $B_{v_j} = \{b_j\}$. 
	For any $r \neq i,j$, since $u_r$ is adjacent to $v_i$ and $v_j$, the set
	$\{b_i,b_j\}$ is a subset of $B_{u_r}$. Therefore, $|B_{u_r}| = n - 1$, 
	and hence, $|A_{u_r}| = 1$, for all $r \in [n]$. 
	By the pigeon hole principle, among the $n$
	vertices $u_1,\ldots,u_n$, there must be two distinct vertices, say $u_r$ and $u_s$, that 
	satisfy $A_{u_r} = A_{u_s}$. Moreover, as $|B_{u_r}| = |B_{u_s}| = n - 1$, 
	we must have $B_{u_r} \cap B_{u_s} \neq \varnothing$ as well.
	We obtain a contradiction, since the cointersection Condition is violated.  
\end{itemize}
Thus, if $n-1$ is an odd prime then $\ttc(\K^M_{n,n}) \neq 2n-1$.
Combining this fact with \eqref{eq:2} and \eqref{eq:KM}, we conclude that $\ttc(\K^M_{n,n}) = 2n$
in this case.
\end{proof} 
An obvious corollary of Proposition~\ref{pr:KMnn} is that there exists
infinitely many bipartite graphs where the lower bound $\ttc \geq
\min_{\al\be \geq \tto} (\al+\be)$ established in Lemma~\ref{lem:lower_bound} 
is not attained. 

\section{Algorithms for the Cointersection Model}

In what follows, we develop two algorithms for finding (exact and approximate) cointersection representations of a graph. 
The first algorithm is based on a transformation to instances of the Satisfiability Problem (SAT) and outputs an \emph{optimal} cointersection representation, which uses exactly $\ttc$ features.  
The second algorithm is based on the well known simulated annealing approach, which produces an \emph{approximate} cointersection representation of a graph. 
More specifically, this algorithm inputs $\G$, $\alpha$, and $\beta$, and outputs feature assignments to all vertices of the graph so as to maximize, as much as possible, the score of the representation, i.e. the number of pairs $(u,v)$ that satisfy the cointersection Condition. 

\subsection{Uniqueness of Optimal Cointersection Representations}
\label{subsec:uniqueness}

Before presenting the two algorithms, we briefly discuss the question of uniqueness of an optimal cointersection representation of a graph. Throughout our analysis, we tacitly assume that $\al \leq \be$ for all $(\al,\be)$-CIRs.

Two cointersection representations are considered \emph{equivalent} if one can be obtained from the other by possibly swapping the set of $\A$-features and the set of $\B$-features (only if
$|\A| = |\B|$), and by permuting features within each set.
A graph is said to be \emph{uniquely cointersectable} if all of its \emph{optimal} cointersection representations are equivalent. 
The issue of unique cointersection representations is of importance in practical applications, where different feature assignment algorithms may construct diverse solutions and where we would like to understand how many different solutions are possible. The related concept of \emph{uniquely intersectable} graphs was studied in~\cite{AlterWang1977, MahadevWang1999}. It was proved in~\cite[Thm.~3.2]{MahadevWang1999} that every 
\emph{diamond-free} graph is uniquely intersectable (more precisely, \emph{uniquely intersectable with respect to a multifamily}). Note that a diamond is obtained by removing one edge in $\mathcal{K}_4$.
The problem of finding a \emph{necessary and sufficient} condition for a graph to be uniquely intersectable is widely open.

Some examples of uniquely cointersectable graphs include: 
\begin{itemize}
	\item Cliques $\K_n$, $n \geq 2$, which have a unique $(1,1)$-CIR with all vertices having features $\{a_1 \mid b_1\}$, 
	\item $\K_n - e$, $n \geq 2$, where $e = (u,v)$ is an arbitrary edge. This graph has a unique $(1,2)$-CIR in which $u$ is 
	assigned the pair of features $\{a_1 \mid b_1\}$, $v$ is assigned $\{a_1 \mid b_2\}$, while all other
	vertices (if any) are assigned the set $\{a_1 \mid b_1, b_2\}$. 
	\item The path $\P_5$ has a unique $(2,2)$-CIR, where the vertices from $1$ to $5$ are respectively assigned the following	sets of features: $\{a_1 \mid b_1\}, \{a_1 \mid b_1, b_2\}, \{a_1, a_2 \mid b_2\}, \{a_2 \mid b_1, b_2\}$,
	and $\{a_2 \mid b_1\}$,
	\item The cycle $\C_4$ has a unique $(2,2)$-CIR, where the vertices from $1$ to $4$ are respectively assigned the following	sets of features: $\{a_1, a_2 \mid b_1\}, \{a_1 \mid b_1, b_2\}, \{a_1, a_2 \mid b_2\}, \{a_2 \mid b_1, b_2\}$. 
\end{itemize}

A graph may not have a unique cointersection representation, even if we restrict ourselves to \emph{optimal} $(\alpha, \beta)$ cointersection representations, where $\alpha$ and $\beta$ are fixed, and $\alpha + \beta = \ttc$. 
An example of two optimal $(2, 3)$-CIRs of the path $\P_7$ that are not equivalent is presented in Fig.~\ref{fig:not_unique}. In fact, we prove in Corollary~\ref{cr:not_unique} that \emph{every} path $\P_n$, $n \geq 4$, except
$\P_5$, is \emph{not} uniquely cointersectable. 
A similar result also holds for cycles, but we omit the proof due to lack of space.  
In fact, most paths have at least exponentially many nonequivalent optimal cointersection representations (Theorem~\ref{thm:path_exp}). 
Note that a path or a cycle, which is obviously diamond free, is always uniquely intersectable. 
These results suggest that uniquely cointersectable graphs are even scarcer than uniquely intersectable ones.
The problem of finding a necessary and/or sufficient condition for a graph to be uniquely cointersectable is also open. 

\begin{theorem} 
\label{thm:path_exp}
Every path $\P_n$ with $n \geq 6$ has at least $(\lceil \sqrt{n-1} \rceil - 1)!$ nonequivalent optimal cointersection representations.
\end{theorem} 
\begin{proof}
The main idea behind the proof is to construct a list of at least $(\lceil \sqrt{n-1} \rceil - 1)!$ optimal cointersection representations 
of $\P_n$, and then show that for every pair of representations, there exist two vertices whose sets of assigned
features intersect in a nonequivalent manner. 

Two nonequivalent optimal $(2,3)$-cointersection representations of $\P_7$ are shown in Fig.~\ref{fig:not_unique}. 
If we delete the last vertex and edge in the paths, we obtain two nonequivalent representations for $\P_6$.   
\vspace{-10pt}
\begin{figure}[htb]
\centering
\includegraphics[scale=0.8]{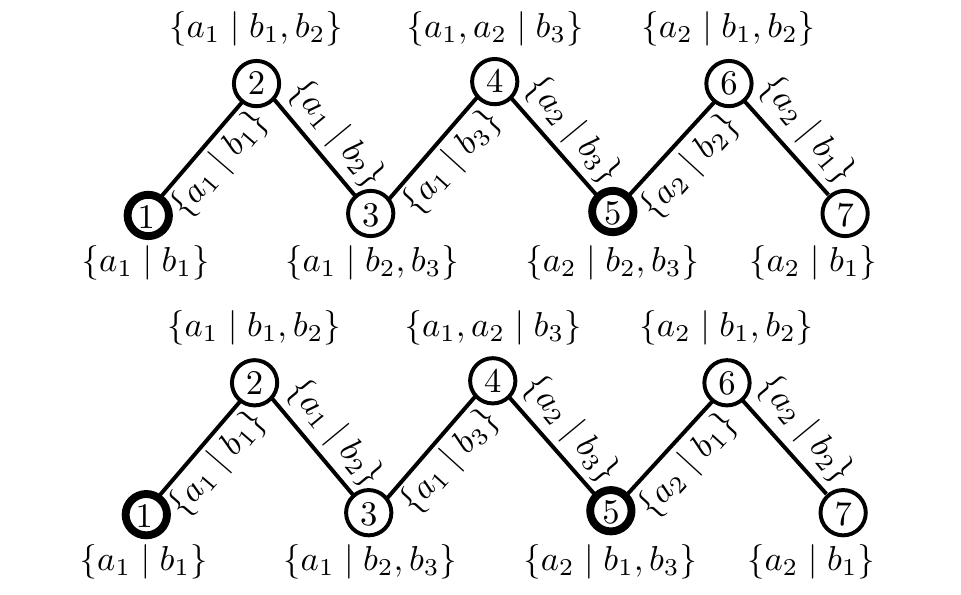}
\caption{An illustration of two nonequivalent, optimal $(2,3)$-CIRs of the path $\P_7$. In the first (top) representation, vertex $1$ and vertex $5$ do not share any features,
while in the second (bottom) representation, they do share one feature, $b_1$.}
\label{fig:not_unique}
\end{figure}

Now suppose that $n \geq 8$ and that we have an optimal $(\al,\be)$-CIR of $\P_n$. If $\be \geq \al + 2$, then $(\al+1)(\be-1) > \al\be$, and hence by Proposition~\ref{pr:simple_graphs}, there is another optimal $(\al+1,\be-1)$-CIR of $\P_n$. We can repeat this argument to obtain an optimal representation with $\al \leq \be \leq \al + 1$ (Note that this argument also reveals that for paths, there always exists a balanced optimal cointersection representation). By Lemma~\ref{lem:lower_bound}, $\al(\al+1) \geq \al\be \geq \tto(\P_n) = n - 1 \geq 7$. Hence, $\beta \geq \al \geq 3$. We also have $\be \geq \lceil \sqrt{n-1} \rceil$. 

We describe next a list of $(\be-1)!$ $(\al,\be)$-cointersection representations of $\P_n$ and proceed to prove
that the representations are pairwise nonequivalent. Each of these representations corresponds to a particular permutation
$\sigma$ of the set $\{1,2,\ldots,\be-1\}$, denoted by $\cR_\sigma$.  
Following the proof of Proposition~\ref{pr:simple_graphs} for paths,
we partition the set of $n-1$ edges into $\al$ groups of $\be$ consecutive edges each, except for possibly the last group, which may contain less than $\be$ edges if $\al\be > n - 1$. In all representations, we assign
$\be$ pairs of features $\{a_1,b_1\}, \{a_1,b_2\},\ldots,\{a_1,b_\be\}$ to the \emph{first} group of $\be$ consecutive edges in that order. 
In the representation $\cR_\sigma$, we continue to assign $\be$ pairs of features $\{a_2,b_\beta\}, \{a_2,b_{\sigma(\be-1)}\}, \{a_2,b_{\sigma(\be-2)}\}, \ldots,\{a_1,b_{\sigma(1)}\}$ to the next group of $\be$ consecutive edges in that order. Similarly, the third group of edges is assigned pairs of features $(a_3,b_{\sigma(1)}), 
(a_3,b_{\sigma(2)}), \ldots$, in $\cR_\sigma$, and so forth.
In general, the rule is to assign different features $a_i$ to different groups of edges, and
to assign the features $b_j$ in such a way that the \emph{last edge} of one group is assigned 
the same $b_j$ as the \emph{first edge} of the following group.   
This process is continued until all edges are assigned
one pair of features each. Upon completion of this procedure, each vertex is assigned the union of the sets of features assigned to its adjacent edges. According to the argument used in the proof of Proposition~\ref{pr:simple_graphs} for paths, 
each $\cR_\sigma$ represents an $(\al,\be)$-cointersection representation of $\P_n$. 

It remains to prove that for two different permutations $\sigma$ and $\sigma'$ of $\{1,2,\ldots,\be-1\}$, there exist two distinct vertices
$u$ and $v$ whose sets of assigned features intersect differently in the two representations. More specifically, $u$ lies within the first group of vertices and $v$ lies within the second group of vertices. 
Let $j \in [\be-1]$ be the largest index satisfying
$z \define \sigma(j) \neq t \define \sigma'(j)$. Then $y \define
\si(j+1) = \si'(j+1)$. Note that if $j = \be-1$, one may set
$y = \beta$. 
Without loss of generality, let us also assume that $t > z$. 
We select $v$ (see Fig.~\ref{fig:uv1} and Fig.~\ref{fig:uv2}) to be the vertex adjacent to the two consecutive edges in the second group which are assigned features $\{a_2,b_y\}$ and $\{a_2,b_z\}$
in $\cR_\si$. In $\cR_{\si'}$, $v$ is adjacent to two edges with assigned features $\{a_2,b_y\}$ and $\{a_2,b_t\}$. 
As $\al \geq 3$, both groups have $\beta$ edges and vertices $u$ and $v$ as described above always exist. 

We consider two cases which correspond to different choices of $u$. It suffices to show that in both cases, $u$ and $v$ have a different number of common features in $\cR_\si$
and $\cR_{\si'}$. 

\textbf{Case 1.} $t = z + 1$. 
We select $u$ (see Fig.~\ref{fig:uv1}) as the vertex adjacent to the two consecutive edges in the \emph{first} group that are assigned features $\{a_1,b_t\}$ and $\{a_1,b_{t+1}\}$ in both $\cR_\si$ and $\cR_{\si'}$. Note that $t \leq \be-1$, and hence
$t+1 \leq \be$. 
\begin{figure}[htb]
\centering
\includegraphics[scale=0.85]{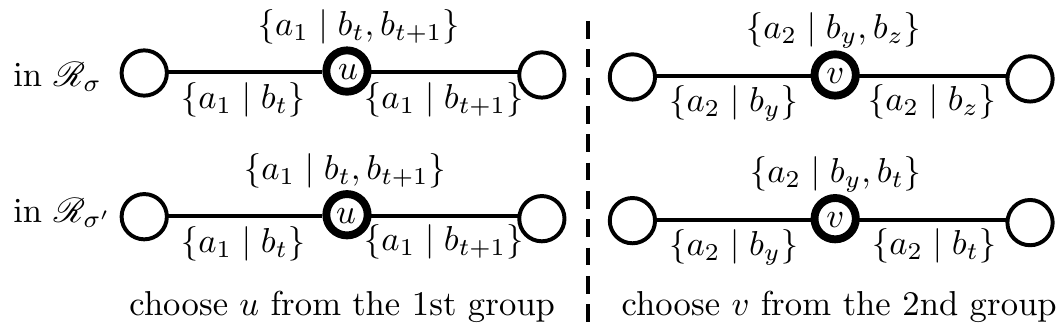}
\caption{(Case 1) The feature sets of $u$ and $v$ with respect to $\cR_\si$ and $\cR_{\si'}$.}
\label{fig:uv1}
\end{figure}
Since $y \notin \{z,t\}$, we consider the following two sub-cases. 
If $y < z$ or $y > t + 1,$ then in $\cR_\si$ the vertices $u$ and
$v$ do not share any features, while in $\cR_{\si'}$, 
they do share one common feature, namely $b_t$. 
If $y =  t + 1,$ then in $\cR_\si$ the vertices $u$ and
$v$ share precisely one feature, namely $b_{t+1}$, 
while in $\cR_{\si'}$, they share two features, $b_t$
and $b_{t+1}$. 

\textbf{Case 2.} $t > z + 1$. 
We select $u$ (see Fig.~\ref{fig:uv2}) as the vertex adjacent to the two consecutive edges in the \emph{first} group that are assigned $\{a_1,b_z\}$ and $\{a_1,b_{z+1}\}$ in both $\cR_\si$ and $\cR_{\si'}$. 
\begin{figure}[htb]
\centering
\includegraphics[scale=0.85]{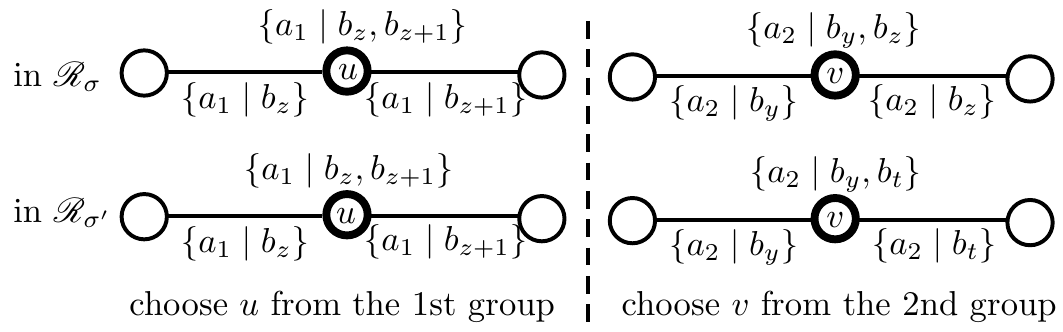}
\caption{(Case 2) The feature sets of $u$ and $v$ with respect to $\cR_\si$ and $\cR_{\si'}$.}
\label{fig:uv2}
\end{figure}
If $y < z$ or $y > z + 1$ then in $\cR_\si$ the vertices $u$ and
$v$ share one feature, namely $b_z$, while in $\cR_{\si'}$, 
they do not share any features. 
If $y =  z + 1,$ then in $\cR_\si$, the vertices $u$ and
$v$ share precisely two features, namely $b_z$ and $b_{z+1}$, 
while in $\cR_{\si'}$, they share only one feature, namely $b_{z+1}$.

This completes the proof.
\end{proof}

\vspace{-3pt}
\begin{corollary}
\label{cr:not_unique}
None of the paths $\P_n$, $n \geq 4$, except for $\P_5$, is uniquely cointersectable. 
\end{corollary} 
\begin{proof} 
By Proposition~\ref{pr:simple_graphs}, $\P_4$ has a $(1,3)$-CIR as well as a $(2,2)$-CIR, both of which are optimal. Hence, $\P_4$ is not uniquely cointersectable. For $n \geq 6$, according to Theorem~\ref{thm:path_exp}, $\P_n$ has at least $2 = (\lceil \sqrt{6-1} \rceil - 1)!$ 
nonequivalent optimal cointersection representations, and is hence not uniquely cointersectable. 
\end{proof}

\subsection{Feature Assignments via SAT Solvers}
\label{subsec:SAT}

For arbitrary $\alpha$ and $\beta$, it is an NP-complete problem to determine if 
an $(\al,\be)$-CIR exists; indeed, when $\al = 1$, the problem becomes 
whether there exists an intersection representation that uses $\be$ features,
which is known to be NP-complete~\cite{Orlin1977}. We discuss below a means of determining
the cointersection number in a constructive manner, which also results in feature assignments for
the vertices. The idea is to restate the cointersection problem as a Satisfiability Problem (SAT). 

Given $\al$, $\be$, and a graph $\G$ on $n$ vertices, we construct an instance
of a SAT problem that is satisfiable if and only if there exists an $(\al,\be)$-CIR of $\G$. 
An optimal pair $(\al,\be)$, therefore, can be determined via a simple binary search.
We use the variables $x_{u,a}$ and $y_{u,b}$, for $u \in [n]$, $a \in [\al]$, 
$b \in [\beta]$, where $x_{u,a}=1$ and $y_{u,b} = 1$ mean that the vertex $u$ is assigned a feature $a \in \A=[\al]$ and a feature $b \in \B = [\be]$,
respectively. For each edge $(u,v)$, we want the formula \vspace{-3pt}
\begin{equation} 
\label{eq:edge1}
\Big( \vee_{a \in [\al]} (x_{u,a} \wedge x_{v,a}) \Big) \wedge
\Big( \vee_{b \in [\beta]} (y_{u,b} \wedge y_{v,b}) \Big) \vspace{-3pt}
\end{equation} 
to be satisfiable, which is equivalent to the requirement that $u$ and $v$ have
some common features $a \in \A$ and $b \in \B$. To turn this formula into a conjunctive 
form, we introduce the variable $A_{u,v,a}$ and add one more requirement that
$A_{u,v,a} \leftrightarrow (x_{u,a} \wedge x_{v,a})$, which stands for \vspace{-3pt}
\begin{equation} 
\label{eq:Auva}
(\overline{A_{u,v,a}} \vee x_{u,a}) \wedge (\overline{A_{u,v,a}} \vee x_{v,a})
\wedge (A_{u,v,a} \vee \overline{x_{u,a}} \vee \overline{x_{v,a}}). \vspace{-3pt}
\end{equation} 
Similarly, we include $B_{u,v,b} \leftrightarrow (y_{u,b} \wedge y_{v,b})$, 
which stands for \vspace{-3pt}
\begin{equation} 
\label{eq:Buvb}
(\overline{B_{u,v,b}} \vee y_{u,b}) \wedge (\overline{B_{u,v,b}} \vee y_{v,b})
\wedge (B_{u,v,b} \vee \overline{y_{u,b}} \vee \overline{y_{v,b}}). \vspace{-3pt}
\end{equation} 
One may hence rewrite \eqref{eq:edge1} as \vspace{-3pt}
\begin{equation} 
\label{eq:edge2}
( \vee_{a \in [\al]} A_{u,v,a}) \wedge ( \vee_{b \in [\beta]} B_{u,v,b}). \vspace{-3pt}
\end{equation} 
If $(u,v)$ is not an edge, we introduce the variables $C_{u,v}$ and $D_{u,v}$
and the following clauses \vspace{-3pt}
\begin{equation}
\label{eq:CD}
\overline{C_{u,v}} \vee \overline{D_{u,v}},  \vspace{-3pt}
\end{equation} 
\begin{equation}
\label{eq:cx}
C_{u,v} \vee \overline{x_{u,a}} \vee \overline{x_{v,a}}, \text{ for every } a\in [\al], \vspace{-3pt}
\end{equation} 
\begin{equation}
\label{eq:dy}
D_{u,v} \vee \overline{y_{u,b}} \vee \overline{y_{v,b}}, \text{ for every } b\in [\be].\vspace{-3pt} 
\end{equation} 
These clauses impose the condition that $u$ and $v$ either have no
common feature in $\A = [\al]$ or have no common feature in $\B=[\be]$.
Using \eqref{eq:Auva}--\eqref{eq:dy}, we can now create an instance of SAT in the conjunctive normal form (CNF), which may be solved by Minisat~\cite{EenSorensson2003}.
The interested reader is referred to~\cite{BergJarvisaloICDMW2013} for a related discussion on intersection representations.  

\begin{example}
\label{ex:SAT_heurisctic}
To create a graph with a ``ground truth'' cointersection representation, we start off by fixing $\alpha = 4$, 
$\beta = 5$, and randomly assign two subsets $A_u \subseteq \{a_1,\ldots,a_4\}$ and 
$B_u \subseteq \{b_1,\ldots,b_5\}$ to each vertex $u \in \V = [11]$. 
The feature sets of the vertices are given in the second and the third columns of 
Table~\ref{tab:truth_SAT_Heuristic}. 
The edges are then created according to the cointersection Condition. The graph is depicted in 
Fig.~\ref{fig:example}. 
\vspace{-10pt}
\begin{figure}[htb]
\centering
\includegraphics[scale=0.9]{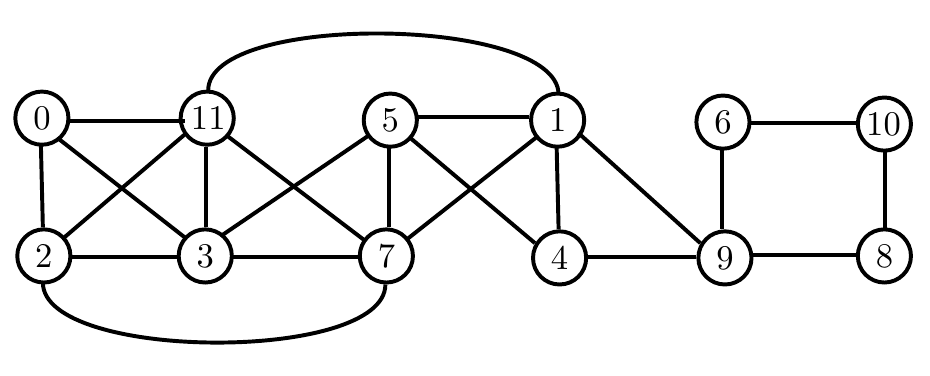}
\caption{A graph on $12$ vertices whose edges are generated by the feature assignment given in 
the second and the third columns of Table~\ref{tab:truth_SAT_Heuristic}, according to the cointersection
Condition.}
\label{fig:example}
\end{figure}
\begin{table}[htb]
\centering
\begin{tabular}{|l||l|l||l|l||}
\hline
Vertex & $A_v$ & $B_v$ & $A^{\text{SAT}}_v$ & $B^{\text{SAT}}_v$\\
\hline
0 & $a_0$ & $b_1$ & $a_2$ & $b_1$\\
\hline
1 & $a_0a_2$ & $b_0b_2$ & $a_0a_1a_2$ & $b_0b_2$\\
\hline
2 & $a_2$ & $b_1b_3$ & $a_2$ & $b_1b_3$\\
\hline
3 & $a_2a_3$ & $b_1b_3$ & $a_2$ & $b_1b_4$\\
\hline
4 & $a_0$ & $b_2b_3$ & $a_0a_1$ & $b_2$\\
\hline
5 & $a_0a_3$ & $b_0b_3$ & $a_1a_2$ & $b_2b_4$\\
\hline
6 & $a_1$ & $b_2b_3$ & $a_0a_1$ & $b_3$\\
\hline
7 & $a_2a_3$ & $b_0b_3b_4$ & $a_2$ & $b_0b_3b_4$\\
\hline
8 & $a_1$ & $b_0b_1$ & $a_0a_1$ & $b_1$\\
\hline
9 & $a_0a_1$ & $b_1b_2$ & $a_0$ & $b_0b_1b_2b_3$\\
\hline
10 & $a_1$ & $b_0b_3b_4$ & $a_1$ & $b_1b_3$\\
\hline
11 & $a_2$ & $b_0b_1b_3$ & $a_2$ & $b_0b_1$\\
\hline
\end{tabular}
\vspace{3pt}
\caption{Two different cointersection representations of the graph depicted in Fig.~\ref{fig:example}.
The sets $A_v$ and $B_v$ correspond to the random feature assignment that generates this graph.
The sets $A^{\text{SAT}}_v$ and $B^{\text{SAT}}_v$ correspond to the output of the SAT-based exact
algorithm developed in Section~\ref{subsec:SAT}.
 }
\label{tab:truth_SAT_Heuristic}
\end{table} 

The exact algorithm based on a SAT solver described in Section~\ref{subsec:SAT} reveals that $\tcG = 8$ and provides a $(3,5)$-cointersection representation as given
in the forth and fifth columns of Table~\ref{tab:truth_SAT_Heuristic}. 
In this case, the representation found by the algorithm turns out to be more compact than 
the ``ground truth'', which is often the case when we test with graphs generated from a random feature assignment.
Note that to visually compare two different representations, we relabel the features of one representation
in a way that maximizes the averaged Jaccard similarity between the sets of features assigned to each vertex
in two representations. Relabeling of $\A$-features and $\B$-features are performed separately. 
Here, the Jaccard similarity between the two sets $S$ and $T$ is defined as $|S \cap T|/|S \cup T|$. 
\end{example}

\subsection{A Simulated Annealing Algorithm for Approximate Cointersection Representation Inference}
\label{subsec:MCMC}

It is important to have approximate cointersection representations of a graph, especially when the 
graph is constructed from a real world data set, where data is usually noisy and an exact representation is, therefore, not necessary. 
Moreover, for large graphs, an approximate representation may still provide insight into the structure of the data, 
without over-representing the graphs with too many features. In this subsection, we present a randomized algorithm 
based on simulated annealing that produces an approximate $(\al,\be)$-cointersection representation of a graph, for any 
fixed pair $(\al, \be)$ given as an input.  
We also illustrate an applications of the algorithm to a real world network and discuss the structure of overlapping communities induced by the output representation which coincides with the ground truth. 
\begin{figure}[htb]
\centering
\fbox{
\parbox{3.4in}{
\centerline{\bf A Randomized Algorithm}
\begin{algorithmic}[1]
\STATE \textbf{Input:} A graph $\G$, integer parameters $\al, \be$, mixing parameter $c$, number of rounds $N$;
\STATE \textbf{Initialization:} 
\begin{itemize}
	\item Assign to all $v \in \V$ nonempty sets $A_v \subseteq \A=\{a_1,\ldots,a_\al\}$
and $B_v \subseteq \B = \{b_1,\ldots,b_\be\}$, chosen uniformly at random; 
	\item Initially, let both $\L$ and $\L_{\max}$ denote the chosen random feature assignments;
\end{itemize}
\REPEAT
\STATE Choose a vertex $u \in \V$ uniformly at random; 
\STATE Select $\varnothing \neq A'_u \subseteq \A$ and $\varnothing \neq B'_u \subseteq \B$ at random;
\STATE Let $\L' \leftarrow \L$ by assigning $A'_u$ and $B'_u$ to $u$;
\STATE Set $\L = \L'$ with probability $\min\{1,e^{c\big(s(\L') - s(\L)\big)}\}$;  
\IF {$\L$ is replaced by $\L'$ and $s(\L') > s(\L_{\max})$}
			\STATE Set $\L_{\max} = \L'$;
\ENDIF
\UNTIL{the loop has run for $N$ rounds;}
\STATE \textbf{Output:} $\L_{\max}$;
\end{algorithmic}
}
}
\caption{A simulated annealing algorithm for determining approximate cointersection representations of graphs.
The score $s(\L)$ counts the number of edges/non-edges of $\G$
that match the feature assignment $\L$ according to the Cointersection Condition.}
\label{fig:MCMC}
\end{figure}

The randomized algorithm (Fig.~\ref{fig:MCMC}) first assigns to each vertex $v \in \V$ a random
set of $\A$-features, namely $A_v$, and a random set of $\B$-features, namely $B_v$, both of which should be nonempty. 
This is referred to as the feature assignment $\L$.
Subsequently, it enters a loop of $N$ rounds, where $N$ is set to $b\,n\log(n)$ with some constant $b$. 
In each round, it chooses a random vertex $u$ and generates two random sets $A'_u \subseteq \A$ and $B'_u \subseteq \B$. 
Let $\L'$ be the feature assignment obtained from $\L$ by replacing $A_u$ by $A'_u$ and $B_u$ by $B'_u$. 
The \emph{score} $s$ of any feature assignment $\L$ is defined as the number of edges/non-edges of the graphs that
match $\L$, according to the Cointersection Condition. 
If $s(\L') > s(\L)$ then we set $\L := \L'$. Otherwise, we do it with probability $e^{c\big(s(\L') - s(\L)\big)}$. 
We usually set $c$ to be a constant, for example, $c = 10$ in our subsequent examples. 
For a more detailed discussion of the role of $c$ in the convergence speed of the underlying Markov chain, 
the reader may refer to the work of Tsourakakis~\cite{TsourakakisWWW2015} on intersection representation of graphs. 
At any time, $\L_{\max}$ records the feature assignment with maximum score seen so far.  

\begin{example}  
\label{ex:Karate}

We consider the social network of friendships among 34 members of an university-based Karate club, introduced by Zachary~\cite{Zachary1977}. Each individual is represented by a node in the
network and two nodes are joined by an edge if and only if the two corresponding individuals were
consistently observed to interact outside the normal activity time of the club (Fig.~\ref{fig:Karate}). As a result of a dispute between the instructor (Node 1) and the club president (Node 34), the members of the clubs were split into two groups, one supporting the president and the other 
supporting the instructor. This fission naturally induced two communities inside the club, corresponding to the aforementioned groups. As 
some form of ``the ground truth'' community structure is known, this data set has become a well known benchmark for community detection algorithms. 

\vspace{-5pt}
\begin{figure}[htb]
\centering
\includegraphics[scale=0.65]{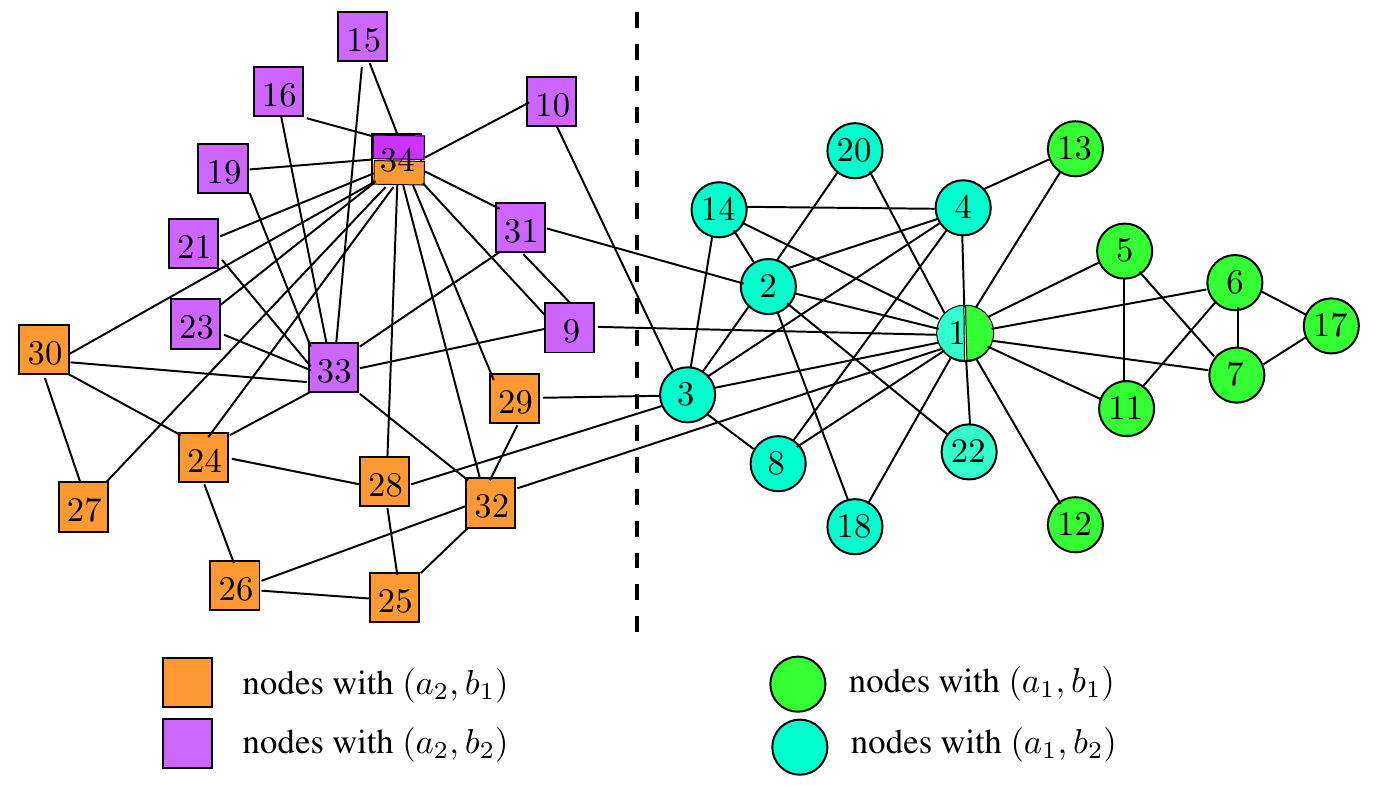}
\caption{The social network of friendships in a Karate club. The members of the club were 
naturally divided into two groups, the one on the left supporting the president (Node 34), and
the one on the right supporting the instructor (Node 1). Given $\al = \be = 2$ as input parameters, the 
randomized algorithm recovered a community structure, with two disjoint communities 
which correspond exactly to the two groups of supporters as discussed. 
But the algorithm provided more information, as within each community two further overlapping sub-communities, marked by different
colors, where identified. The only overlapping was in terms of Node 34 and Node 1, marked with a mix of two colors, correspond to the club president and the instructor. This suggests that there were two sub communities within each community held together by the president and the instructor.}
\label{fig:Karate}
\end{figure}

Applying the randomized algorithm to this network, with $\al = \be = 2$, a community structure
is revealed as illustrated in Fig.~\ref{fig:Karate}. The set of nodes with feature $a_1$ corresponds to the supporters of the instructor (Node 1), while the set of nodes with feature $a_2$ corresponds to the supporters of the club president (Node 34). 
Each of these two sets is further divided into overlapping sub-communities, marked by different
colors, where the overlapping nodes, marked with a mix of two colors, correspond to the club president and the instructor. 
Thus, in this case, the algorithm produces an ``error-free'' result if we look at communities
defined via features $a_1$ and $a_2$. We refer to these as the $\A$-communities. 

The communities induced by the $\B$-features, referred to as the $\B$-communities, are $\{1,5,6,7,11,12,13,17,24,25,26,27,28,29,30,32,34\}$ and $\{1,2,3,4,8,9,10, 14, 15, 16, 18, 19, 20, 21, 22, 23, 31, 33, 34\}$. Each of these communities comprises nodes from both two $\A$-communities. This result is a consequence of the way we define the cointersection model: it still allows one to identify shared features of individuals not necessarily within the same community; furthermore, if $\al = \be = 2$, the community structure dictated by the randomized algorithm usually looks like an overlapping \emph{grid}, as shown in Fig.~\ref{fig:cir}. Each greed may define communities of potentially different relevance; if the dataset does not have a grid-like community structure, some communities detected by the algorithm may require more elaborate interpretations.  

\vspace{-10pt}
\begin{figure}[htb]
\centering
\includegraphics[scale=0.7]{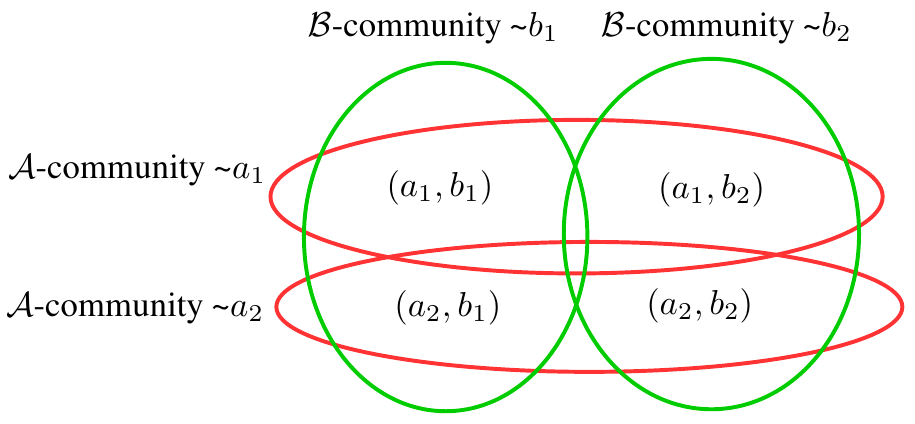}
\caption{Rough sketch of a typical community structure induced by an approximate $(2,2)$-cointersection representation of a graph. There are $4 = \al + \be$ (possibly overlapping) communities corresponding to vertices that are assigned a particular $\A$- or $\B$-feature. There are also 
$4 = \al\be$ (possibly overlapping) tighter-knit sub-communities, each of which consists of nodes that are assigned a particular pair of features $(a_i,b_j)$, where $a_i \in \A$ and $b_j \in \B$.}
\label{fig:cir}
\end{figure}
\end{example} 

\begin{remark}
Note that if we set $\al = 1$ then the randomized algorithm coincides with the 
algorithm developed in Tsourakakis's work~\cite{TsourakakisWWW2015} for intersection
representation. In Example~\ref{ex:Karate}, if we set $\al  = 1$ and $\be = 2$, then the algorithm
also outputs two communities that correspond perfectly to the ground truth.    
\end{remark}

\section{Extension to General Boolean Functions}
\label{sec:Boolean}

We extend the bounds developed for the cointersection model
in Section~\ref{sec:upper_bounds}, which is based on the AND Boolean function, to cater to models based on more general Boolean functions.

Let $f=f(x_1,x_2,\ldots,x_r)$ be a Boolean function in the \emph{full} disjunctive normal form. In other words, the corresponding logical formula
of the Boolean function is a disjunction $(\vee)$ of one or more conjunctions 
$(\wedge)$ of one or more literals, where each variables appears exactly once
in every clause. Some examples are $f(x_1, x_2, x_3) = x_1 \vee (x_2 \wedge x_3)$ and $f(x_1, x_2, x_3, x_4) = (x_1 \wedge x_2) \vee (\neg x_1 \wedge x_3 \wedge x_4)$. 
We first discuss the meanings of the AND $(\wedge)$ operator, the 
OR $(\vee)$ operator, and the NEGATION $(\neg)$ operator, and then proceed
to describe the model corresponding to a general Boolean function in its 
full disjunctive normal form. 
 
\emph{The AND function $f(x_1,x_2) = x_1 \wedge x_2$.} Let $\A^1$ and $\A^2$
be two pairwise disjoint nonempty sets of features of cardinalities $\alpha_1$ and $\alpha_2$, respectively.
In an $(\al_1 \mid \al_2)$-AND-intersection representation of a graph 
$\G = (\V, \E)$, each vertex $v \in \V$ is assigned
two sets $A^i_v \subseteq \A^i$, $i \in [2]$, such that for every $u \neq v$, $u, v\in \V$,
it holds that $(u,v) \in \E$ if and only if $A^1_u \cap A^1_v \neq \varnothing$ and
$A^2_u \cap A^2_v \neq \varnothing$. The AND-intersection number of $\G$ is the smallest number of features used, i.e. $\al_1+\al_2$, in any $(\al_1 \mid \al_2)$-AND-intersection representation of the graph. The AND-intersection number of $\G$ is precisely the cointersection number of the graph.  

\emph{The OR function $f(x_1,x_2) = x_1 \vee x_2$}. 
Let $\A^1$ and $\A^2$ be two pairwise disjoint nonempty sets of features of cardinalities $\alpha_1$ and $\alpha_2$, 
respectively.
In an $(\al_1 \mid \al_2)$-OR-intersection representation of a graph 
$\G = (\V, \E)$, each vertex $v \in \V$ is assigned
two sets $A^i_v \subseteq \A^i$, $i \in [r]$, such that for every $u \neq v$, $u, v\in \V$,
it holds that $(u,v) \in \E$ if and only if $A^1_u \cap A^1_v \neq \varnothing$ or 
$A^2_u \cap A^2_v \neq \varnothing$. The OR-intersection number of $\G$ is the smallest number of features used, i.e. $\al_1 + \al_2$, in any $(\al_1 \mid \al_2)$-OR-intersection representation of the graph. 
Note that as $\A^1$ and $\A^2$ are disjoint, we can simply let $\A = \A^1 \cup \A^2$,
$\alpha = \al_1+\al_2$, and for each vertex $v$, let $A_v = A^1_v \cup A^2_v$. 
Then an $(\al_1 \mid \al_2)$-OR-intersection representation of
$\G$ simply corresponds to a way to assign to each vertex $v$ a set $A_v \subseteq \A$ of features such that for every $u \neq v$, $u, v\in \V$, it holds that $(u,v) \in \E$ if and only if $A_u \cap A_v \neq \varnothing$. 
This is precisely the definition of an intersection representation of $\G$. 
Thus, the OR-intersection number of a graph is the same as its intersection number, as long as the intersection number is at least two. 
  
\emph{NEGATION function $f(x) = \neg x$}. Let $\A$ be a nonempty set of features of cardinality $\alpha$. In an $(\alpha)$-NEGATION-intersection representation, each vertex $v \in \V$ is assigned
a set $A_v \subseteq \A$ such that for every $u \neq v$, $u, v\in \V$, it holds that $(u,v) \in \E$ if and only if $A_u \cap A_v = \varnothing$. 
The NEGATION-intersection number of $\G$ is the smallest number of features $\alpha$ used in any
$(\al)$-NEGATION-intersection representation of $\G$. It is immediate that this number is the same
as the intersection number of the complement of $\G$. 

Suppose we have a general Boolean function $f = f(x_1,x_2,\ldots,x_r)$ written in the
full disjunctive normal form, which involves three operators $\vee$, $\wedge$, and $\neg$. 
Let $\A^1, \A^2,\ldots, \A^r$ be disjoint sets of features of cardinalities $\al_1, \al_2,
\ldots,\al_r$, respectively. In an $(\al_1 \mid \al_2 \mid \cdots \mid \al_r)$-$f$-intersection representation of $\G$, each vertex $v \in \V$ is assigned $r$ sets $A^i_v
\subseteq \A^i$, $i \in [r]$, such that for every $u \neq v$, $u, v\in \V$, it holds that
$(u,v) \in \V$ if and only if the intersections 
of the sets $A^i_u$ and the sets $A^i_v$ follow the rule set by the propositional formula of $f$. 
For example, when $f(x_1,x_2,x_3) = x_1 \vee (x_2 \wedge x_3)$, it is required that 
$(u,v) \in \E$ if and only if the following statement is satisfied. 
\[
(A^1_u \cap A^1_v \neq \varnothing) \vee \Big((A^2_u \cap A^2_v \neq \varnothing)
\wedge (A^3_u \cap A^3_v \neq \varnothing)\Big).
\] 
In words, $u$ and $v$ are adjacent if and only if they share either an $\A^1$-label
or both an $\A^2$-label and an $\A^3$-label. 
For another example, take $f(x_1, x_2, x_3, x_4) = (x_1 \wedge x_2) \vee (\neg x_1 \wedge x_3 \wedge x_4)$. Then in a corresponding representation of $\G$, two vertices are adjacent if and only if either of the following two cases happens: (1) they share both an $\A^1$-label and an $\A^2$-label; or (2) they do not share any $\A^1$-label, but they share both an 
$\A^3$-label and an $\A^4$-label. 
The \emph{$f$-intersection number} of $\G$ is defined to be the smallest number of features used, namely $\sum_{i = 1}^r \al_i$, in any $(\al_1 \mid \al_2 \mid \cdots \mid \al_r)$-$f$-intersection representation of the graph.  

It is not immediately clear that the negation function has sufficiently strong relevance as the AND and OR functions in the context of social network analysis. Hence, we focus on Boolean functions that involve $\vee$ and $\wedge$ operations only and provide the following proposition 
generalizing Lemma~\ref{lem:lower_bound}.
\begin{proposition}
\label{pro:lower_bound_general}
Let $f = f(x_1,x_2,\ldots,x_r)$ be a Boolean function in the full disjunctive normal 
consisting only of $\vee$ and $\wedge$. Let $g_f = g_f(\al_1,\al_2,\ldots,\al_r)$ be an
integer-valued function on $r$ non-negative integral variables $\al_1,\al_2,\ldots,\al_r$, obtained from $f$ by replacing $x_i$ by $\al_i$ $(i \in [r])$, $\vee$ by $+$, and $\wedge$ by $\times$.  
Then the $f$-intersection number of a graph $\G$ is bounded from below by the 
optimal value of the objective function of the integer programming problem given below: 
\[
\begin{split}
\text{(IP)} \qquad &\text{minimize}\quad \sum_{i = 1}^r \al_i \\
&\text{subject to}\quad g_f(\al_1,\al_2,\ldots,\al_r) \geq \toG,\\
&\qquad \quad \qquad \mathbb{Z}\ni \al_i \geq 1, \forall i \in [r].
\end{split}
\]
\end{proposition}
\begin{proof} 
Suppose that we have an $(\al_1 \mid \al_2 \mid \cdots \mid \al_r)$-$f$-intersection representation of the graph $\G$ 
with the corresponding sets of labels $\A^1, 
\A^2,\ldots, \A^r$. For any clause $x_{i_1}\wedge x_{i_2} \wedge \cdots \wedge x_{i_s}$
of $f$, a tuple $(a^{i_1}, a^{i_2}, \ldots, a^{i_r})$ where $a^{i_j} \in \A^{i_j}$
corresponds to a clique in $\G$, which consists of all vertices $v \in \V$
that have $a^{i_1}, a^{i_2}, \ldots, a^{i_r}$ in their feature sets. 
Note that there are in total $g_f(\al_1, \al_2,\ldots,\al_r)$ such cliques.
As each edge of $\G$ must belong to one of these cliques, these cliques form
an edge clique cover of $\G$. Therefore, $g_f(\al_1, \al_2,\ldots,\al_r) \geq \toG$.
\end{proof} 

If we ignore the condition that $\al_i \in \mathbb{Z}$ in the integer programming problem (IP) 
stated in Proposition~\ref{pro:lower_bound_general}, we obtain a real-valued programming 
problem, referred to as (P). An optimal solution to (P) also provides a lower bound
on the $f$-intersection number of the graph. Generally, we can find necessary conditions for a solution of (P) to exist by using either
the method of Lagrange multipliers or the Karush-Kuhn-Tucker (KKT) conditions. 
We illustrate this observation with the following example. 

\begin{example}
\label{ex:general}
Let $f(x_1,x_2,x_3) = x_1 \vee (x_2 \wedge x_3)$. 
Using the notation in Proposition~\ref{pro:lower_bound_general}, $g_f(\al_1,\al_2,\al_3)
= \al_1+\al_2\al_3$. Then the optimal value of the objective function of the following programming problem serves as a lower bound for the $f$-intersection number of a graph $\G$:
\[
\begin{split}
\text{(P)} \qquad &\text{minimize}\quad \al_1+\al_2+\al_3 \\
&\text{subject to}\quad \al_1+\al_2\al_3 \geq \toG,\\
&\qquad \quad \qquad \mathbb{R} \ni \al_i \geq 1, \forall i \in [3].
\end{split}
\]
In order to use the method of Lagrange multipliers,
we first introduce the slack variables $\beta_i$, $i \in [4]$, to convert the inequality
constraints into equality constraints as follows. The constraint $\al_i \geq 1$
is converted into the new constraint $\al_i - \beta_i^2 - 1 = 0$, for each $i \in [3]$, 
and the constraint $\al_1+\al_2\al_3 \geq \theta_1$ is converted into the new constraint
$\al_1+\al_2\al_3 - \beta_4^2 - \theta_1 = 0$. Let $\lam_i$, $i \in [4]$, be the Lagrange multipliers. We formulate the Lagrangian 
\[
\begin{split}
&\L(\al_1, \al_2, \al_3, \beta_1,\beta_2, \beta_3, \beta_4, \lam_1, \lam_2, \lam_3, \lam_4)\\
&= \sum_{i = 1}^3 \al_i + \sum_{i = 1}^3 \lam_i(\al_i - \beta_i^2 - 1) 
+ \lam_4(\al_1+\al_2\al_3 - \beta_4^2 - \theta_1).\\
\end{split}
\]
The method of Lagrange multipliers states that if we examine all stationary points of the
Lagrangian, at which $\nabla \L = \mathbf{0}$, where
$\nabla \L = (\frac{\partial \L}{\partial \al_1}, \ldots, 
\frac{\partial \L}{\partial \al_3}, \frac{\partial \L}{\partial \beta_1},\ldots,
\frac{\partial \L}{\partial \beta_4},\frac{\partial \L}{\partial \lam_1},\ldots,
\frac{\partial \L}{\partial \lam_4})$, then the one that leads to the minimum objective
value $\sum_{i=1}^3\al_i$ is an optimal solution to (P).     
Therefore, using this method, we arrive at the following system $\nabla \L = \mathbf{0}$ of equations:
\begin{subequations} 
\label{eq:Lagrangian}
\begin{numcases}{}
1 + \lam_1 + \lam_4 = 0, \label{seq:a}\\
1 + \lam_2 + \lam_4\al_3 = 0, \label{seq:b}\\
1 + \lam_3 + \lam_4\al_2 = 0, \label{seq:c}\\
\lam_i\beta_i = 0, \quad i \in [4], \label{seq:d}\\
\al_i - \beta_i^2 - 1 = 0, \quad i \in [3], \label{seq:e}\\
\al_1 + \al_2\al_3 - \beta_4^2 - \theta_1 = 0. \label{seq:f}
\end{numcases}
\end{subequations}
A straightforward way to obtain all the solutions of the system~\eqref{eq:Lagrangian}
is by examining all 16 cases, each of which captures whether $\lam_i = 0$ or 
$\beta_i = 0$, $i \in [4]$ (from \eqref{seq:d}). We can ignore certain cases due to symmetry.
As a consequence, we find that the objective function $\sum_{i = 1}^3\al_i$ 
is minimized when $\al_1 = 1$ and $\al_2 = \al_3 = \sqrt{\theta_1-1}$, 
which gives us the lower bound $1 + 2\sqrt{\theta_1-1}$ on the $f$-intersection number of $\G$. 

Another example we considered is $f = (x_1 \wedge x_2) \vee (x_1 \wedge x_3) \vee (x_2 \wedge x_3)$. 
Again, applying the method of Lagrange multipliers and Proposition~\ref{pro:lower_bound_general}, it may be
shown that the $f$-intersection number of $\G$ is at least $\sqrt{3\theta_1}$.  
\end{example}

An upper bound on the $f$-intersection number of a graph of bounded degree, where $f$ only
involves the $\vee$ and $\wedge$ operations, may be obtained in the same way as that for the cointersection number, in Theorem~\ref{thm:bounded_degree}. We present this fact below. 

\begin{theorem} 
\label{thm:bounded_degree_AND}
Let $\G$ be a graph on $n$ vertices with $\Delta(\G) \leq d$. 
Let $f = f(x_1,x_2,\ldots,x_r)$ be a Boolean function in the full disjunctive normal 
consisting of only $\vee$ and $\wedge$.
Let $s$ be the largest number of literals that appear in any clause of $f$. 
Then the $f$-intersection number of $\G$ is at most $c(d, r, s)n^{1/s} + r-s$, where $c(d, r, s)$ is a function of $d$, $r$, and $s$.  
\end{theorem} 
\begin{proof} 
We can assume that no clause $\C'$ of $f$ is a sub-clause of another clause $\C$ (i.e., that all of
the literals of $\C'$ also appear in $\C$), as otherwise we can always remove $\C'$ and obtain an equivalent formula of $f$. 

Now let $\C$ be a clause of $f$ with $s$ literals, referred to as the \emph{leading} clause. Relabeling the indices if necessary, 
we can assume that $\C = \wedge_{i=1}^s x_i$. Let $A^1, \ldots, A^r$ be $r$ pairwise disjoint
sets of features such that $\al_i \define |A^i| = c'(d, r, s)n^{1/s}$ for $i \in [s]$, while 
$\al_j \define |A^j| = 1$
for all $j > s, j \in [r]$. Here $c'(d, r, s)$ is a function of $d$, $r$, and $s$, which will be determined later.
Similar to the proof of Theorem~\ref{thm:bounded_degree}, we show that there exists an
$(\al_1 \mid \al_2 \mid \cdots \mid \al_r)$-$f$-intersection representation of $\G$ by invoking 
the Lov\'{a}sz Local Lemma~\cite{Lovasz1968}. As a consequence, the $f$-intersection number of 
$\G$ is at most $c(d,r,s)n^{1/s} + r - s$, where $c(d, r, s) \define sc'(d, r, s)$.    

We independently assign to every edge $e$ of $\G$ a randomly chosen set of features
$\{a^1(e),a^2(e),\ldots,a^s(e)\}$. Note that we do not assign to $e$ any label $a^j \in \A^j$, 
for $j > s$. For every vertex $v \in \V$ and for every $i \in [r]$, let
\[
A^i_v = \{a^i(e): e = (u,v) \in \E\}.
\]    
Then $A^j_v = \varnothing$ for $j > s$. Hence, $A^j_u \cap A^j_v = \varnothing$ for every $u \neq v$ and $j > s$. Moreover, we know that for any clause $\C' \neq \C$, there must exist a $j > s$ such that $\C'$ contains $x_j$, for otherwise, $\C'$ would be a sub-clause of $\C$.  
Therefore, this feature assignment is an $f$-intersection representation of $\G$ if and
only if for every $u \neq v$, $u, v\in \E$, it holds that
\begin{equation} 
\label{eq:boolean_condition}
(u,v) \in \V \Longleftrightarrow A^i_u \cap A^i_v \neq \varnothing, \text{ for all } i \in [s].
\end{equation} 
In other words, we can focus only on the leading clause $\C = \vee_{i = 1}^s x_i$ and ignore
all other clauses of $f$. 

It is clear that \eqref{eq:boolean_condition} is satisfied for all pairs $(u,v) \in \E$.
We now define for each pair $(u,v) \notin \E$ a bad event $E_{u,v}$ where $A^i_u \cap A^i_v
\neq \varnothing$ for all $i \in [s]$. The goal is to show that
there exists a function $c'(d,r,s)$ of $d$, $r$, and $s$, so that $PD \leq 1/4$, where
$\text{Prob}(E_{u,v}) \leq P$ and each bad event is dependent on at most $D$ other bad events.
Then by the Lov\'{a}sz Local Lemma~\cite{Lovasz1968}, we may conclude that there exists a 
way to assign features to the edges of $\G$ that leads to an $f$-intersection representation 
of $\G$. Just as in the proof of Theorem~\ref{thm:bounded_degree}, we have
\[
\begin{split}
\text{Prob}(E_{u,v}) &= \prod_{i=1}^s \text{Prob}(A^i_u \cap A^i_v \neq \varnothing)\\
&\leq P \define \left(\dfrac{d^2}{\al_i - d + 1}\right)^s
= \dfrac{d^{2s}}{\big(c'(d, r, s)n^{1/s} - d + 1\big)^s}.
\end{split}
\]
We also have $D = 2n(d+1)$. It is straightforward to verify that for $c'(d,r,s) \define (8d^{2s+2})^{1/s} + d - 1$, we have $PD \leq 1/4$.  
\end{proof}

\section*{Acknowledgment}
The authors thank Gregory J. Puleo and Charalampos Tsourakakis for helpful 
discussions. Part of the results will be presented at ISIT 2016.
This work was funded by NIH BD2K Grant 1U01CA198943-01 and NSF Grant IOS 1339388 and CCF 11-17980 and NSF 239 SBC Purdue 41010-38050. 
\vspace{-0.1in}
\bibliographystyle{IEEEtran}
\bibliography{CoIntersectionRepresentation}

\end{document}